\documentclass[a4paper,11pt,reqno ]{amsart}

\usepackage[T1]{fontenc}
\usepackage[utf8]{inputenc}
\usepackage{lmodern}
\usepackage[english]{babel}

\usepackage[%
left=2.5cm,       
right=2.5cm,      
top=3.5cm,        
bottom=3.5cm,     
heightrounded,    
bindingoffset=0mm 
]{geometry}

\usepackage{amsrefs}
\usepackage{amsmath}
\usepackage{amssymb}
\usepackage{mathtools}
\usepackage{mathrsfs}
\usepackage[colorlinks=true,linkcolor=blue,urlcolor=blue,citecolor=blue]{hyperref}

\numberwithin{equation}{section}

\usepackage{hyperref} 

\usepackage{tikz}
\usepackage{graphicx}
\usepackage{xcolor}
\usepackage[font=small]{caption}

\usepackage{bookmark}
\usepackage{esint}

\allowdisplaybreaks

\theoremstyle{plain}
\newtheorem{theorem}{Theorem}[section]
\newtheorem{lemma}[theorem]{Lemma}
\newtheorem{proposition}[theorem]{Proposition}
\newtheorem{corollary}[theorem]{Corollary}
\newtheorem*{mtheorem}{Main Theorem}
\theoremstyle{definition}
\newtheorem{definition}[theorem]{Definition}

\newtheorem{remark}[theorem]{Remark}
\newtheorem{open.problem}[theorem]{Open Problem}

\newcommand{\Leb}[1]{\mathscr{L}^{#1}} 

\newcommand{\N}{\mathbb{N}}
\newcommand{\R}{\mathbb{R}}

\newcommand{\eps}{\varepsilon}

\title[Gaussian Gamma-convergence]{Gamma-convergence of fractional Gaussian perimeter}
\author[A.\ Carbotti]{Alessandro Carbotti}
\address{Dipartimento di Matematica
	e Fisica ``E. De Giorgi'', Universit\`a del Salento,
	Via Per Arnesano, 73100 Lecce, Italy.}
\email{alessandro.carbotti@unisalento.it}
\author[S.\ Cito]{Simone Cito}
\address{Dipartimento di Matematica
	e Fisica ``E. De Giorgi'', Universit\`a del Salento,
	Via Per Arnesano, 73100 Lecce, Italy.}
\email{simone.cito@unisalento.it}
\author[D. A. \ La Manna]{Domenico Angelo La Manna}
\address{Department of Mathematics and Statistics, P.O.\ Box 35 (MaD), FI-40014, University of Jyv\"askyl\"a, Finland.}
\email{domenico.a.lamanna@jyu.fi}
\author[D.\ Pallara]{Diego Pallara}
\address{Dipartimento di Matematica
	e Fisica ``E. De Giorgi'', Universit\`a del Salento, and INFN, Sezione di Lecce,
	Via Per Arnesano, 73100 Lecce, Italy.}
\email{diego.pallara@unisalento.it}

\date{\today}  



\begin{document}
	\begin{abstract}
We prove the $\Gamma$-convergence of the renormalised fractional Gaussian $s$-perimeter to the 
Gaussian perimeter
as $s\to 1^-$. Our definition of fractional perimeter comes from that of the fractional powers of 
Ornstein-Uhlenbeck operator given via Bochner subordination formula. As a typical feature of the 
Gaussian setting, the constant appearing in front of the $\Gamma$-limit does not depend on the dimension.
	\end{abstract}
	
	\maketitle
	
	\tableofcontents

\section{Introduction}\label{sec:intro}

For $s\in(0,1)$, fractional $s$-perimeters in the Euclidean space have been introduced in the seminal 
paper by \cite{CafRoqSav} to study nonlocal minimal surfaces of fractional type, while a generalised 
notion of nonlocal perimeter defined through a positive, compactly supported radial kernel has been  
introduced in \cite{MazRosTol}. In the last years fractional perimeters have been object of many 
studies in relation with fractal sets \cite{lombardini}, phase transitions \cite{valdinoci}, 
and nonlocal mean curvature flows \cite{ChaMorPon}, see also the recent survey \cite{dipierro}. One can 
think of fractional perimeters as the sum of local interactions of a measurable set $E$ with its 
complement $E^c$ in a fixed smooth open and connected set $\Omega$ plus a nonlocal contribution coming 
from the interaction between points in $\Omega$ and in $\Omega^c$. Namely
\begin{equation}
\label{eq:euclideansperimeter}
\begin{split}
P_s(E;\Omega)&=L_s(E\cap\Omega,E^c\cap\Omega)+\big(L_s(E\cap\Omega,E^c\cap\Omega^c)
+L_s(E\cap\Omega^c,E^c\cap\Omega)\big) 
\\
&:=P^L_s(E;\Omega)+P^{NL}_s(E;\Omega),
\end{split}
\end{equation}
where, for any $s\in(0,1)$ and for any $A,B$ measurable and disjoint sets we set
$$
L_s(A,B):=\int_A\int_B\frac{dxdy}{|x-y|^{N+s}}.
$$
The functional $L_s$ denotes the interaction between $A$ and $B$ driven by the fractional singular kernel 
$|x-y|^{-N-s}$ that arises from the Bochner subordination formula for fractional powers of second order 
positive definite linear elliptic operators (see e.g. \cite{MarSan}) through the formula
\begin{equation}
\label{eq:derivkernel}
\frac{C_{N,s}}{|x-y|^{N+s}}=\int_0^{\infty}\frac{H_t(|x-y|)}{t^{\frac s2+1}}dt,
\end{equation}
where $C_{N,s}=\frac{2^s}{\pi^{N/2}}\Gamma\left(\frac{N+s}{2}\right)$ and $H_t$ denotes the 
Gauss-Weierstrass kernel
\begin{equation}\label{GWkernel}	
H_t(r)=\frac{1}{(4\pi t)^{N/2}}e^{-\frac{r^2}{4t}}.
\end{equation}
We notice that when $s\to 1^-$ the local part $P^L_s(E;\Omega)$ goes always to infinity unless 
$E\subset\Omega^c$ or $\Omega\subset E$, as observed in \cite{brezis}. This fact suggests to 
renormalise appropriately the functional in order to have a finite pointwise limit as $s\to 1^-$ as 
shown in \cite{CafVal}, where the authors prove that renormalising by the factor $1-s$, 
the $s$-fractional perimeter converges to the perimeter in the sense of De Giorgi when $s\to 1^-$. 
See also \cite{DiFiPaVa} for the limiting behaviour of the $s$-perimeter as $s\to 0^+$. These 
results follow the approximation of local energies by nonlocal ones proved in 
\cite{BouBreMir, davila, MazSha, ponce}. Moreover, in \cite{AmDeMa} the authors show that 
$(1-s)P_s(E;\Omega)$ approaches the perimeter of $E$ in $\Omega$ even in the $\Gamma$-convergence 
sense as $s\to 1^-$. A similar result has been obtained in \cite{BerPal} for more general kernels 
but with different growth. See also \cite{AlbBel, SavVal} for some applications in phase transitions, 
and \cite{CaDoPaPi} for a partial result including kernels with the same growth as the fractional one 
in the more general setting of Carnot Groups.

Fractional perimeters can be equivalently defined by minimising a Dirichlet energy associated 
with an extension problem for the fractional Laplacian as proved by Caffarelli and Silvestre in 
\cite{CafSil}. This result has been generalised by Stinga and Torrea in \cite{StiTor} for 
fractional powers of more general operators. This last extension has been used in \cite{NovPalSir} 
in order to introduce a fractional Gaussian perimeter in the more general setting of abstract Wiener 
spaces and to prove that the halfspace is the unique minimiser among all sets with prescribed Gaussian 
measure as proved for the Gaussian perimeter in \cite{borell, CarKer, EhrScand, ehrhard, SudCir}. In \cite{CaCiLaPa} the same 
authors of this paper prove the related stability estimate for the fractional Gaussian
isoperimetric inequality in finite dimension.

A different notion of fractional Gaussian perimeter has been given in \cite{DL}
\begin{align}
\label{eq:derosalmanna}
\mathcal{J}^\gamma_s(E;\Omega):=&\int_{E\cap\Omega}e^{-\frac{|x|^2}{4}}dx
\int_{E^c\cap\Omega}\frac{e^{-\frac{|y|^2}{4}}}{|x-y|^{N+s}}dy 
\\ \nonumber
&+\int_{E\cap\Omega}e^{-\frac{|x|^2}{4}}dx\int_{E^c\cap\Omega^c}\frac{e^{-\frac{|y|^2}{4}}}{|x-y|^{N+s}}dy
+\int_{E\cap\Omega^c}e^{-\frac{|x|^2}{4}}dx\int_{E^c\cap\Omega}\frac{e^{-\frac{|y|^2}{4}}}{|x-y|^{N+s}}dy,
\end{align}
where the authors prove that, after rescaling by $(1-s)$, the functional $\mathcal{J}^\gamma_s(E;\Omega)$
approaches the Gaussian perimeter in the $\Gamma$-convergence sense as $s\to 1^-$.

In this paper we define the following fractional Gaussian perimeter 
\begin{equation}
\label{eq:fracgaussperimeter}
\begin{split}
P^\gamma_s(E;\Omega)&:=\int_{E\cap\Omega}d\gamma(x)\int_{E^c\cap\Omega}K_s(x,y)d\gamma(y) 
\\
&+\int_{E\cap\Omega}d\gamma(x)\int_{E^c\cap\Omega^c}K_s(x,y)d\gamma(y)
+\int_{E\cap\Omega^c}d\gamma(x)\int_{E^c\cap\Omega}K_s(x,y)d\gamma(y), 
\end{split}
\end{equation}
where $\gamma$ is the standard Gaussian measure in $\R^N$, whose definition will be recalled in the next section, and the kernel $K_s$ is defined in \eqref{defK_s}. 
The definition in \eqref{eq:fracgaussperimeter} is equivalent to the one given in \cite{CaCiLaPa} when 
$\Omega=\R^N$, it is analogous to \eqref{eq:euclideansperimeter}, in the sense that it depends on a 
fractional kernel $K_s$ defined in terms of an explicit heat kernel as in \eqref{eq:derivkernel} and 
it is not equivalent to \eqref{eq:derosalmanna}, see \eqref{eq:kernelquotient}. 

We notice that in the Gaussian setting no definition of fractional perimeter can satisfy the 
translation invariance property (vi) in the axiomatic treatment proposed in \cite[Pag. 29]{ChaMorPon}, 
as the Gaussian weight $\gamma$ is not translation invariant. 

Inspired by \cite{AmDeMa,DL}, the main result of this paper is the proof of the $\Gamma$-convergence 
of our (renormalised) fractional Gaussian perimeter to the Gaussian perimeter as $s\to 1^-$.

\begin{mtheorem}[\textbf{$\Gamma$-convergence}]
	\label{th:maintheorem}
	For every measurable set $E\subset\R^N$ we have 
	\begin{equation}
	\label{eq:gammaliminf}
	L^1_{\text{\rm loc}}-\Gamma-\liminf_{s\to 1^-}(1-s)P_s^{\gamma,L}(E;\Omega)\ge 
	\frac{\sqrt 2}{\pi} P^\gamma(E;\Omega)
	\end{equation}
	and
	\begin{equation}
	\label{eq:gammalimsup}
	L^1_{\text{\rm loc}}-\Gamma-\limsup_{s\to 1^-}(1-s)P_s^{\gamma}(E;\Omega)\le  
	\frac{\sqrt 2}{\pi} P^\gamma(E;\Omega).
	\end{equation}
\end{mtheorem}

We recall that \eqref{eq:gammaliminf} means that 
$$
\liminf_{n\to\infty}(1-s_n)P_{s_n}^{\gamma,L}(E_n;\Omega)\ge \frac{\sqrt{2}}{\pi}P^\gamma(E;\Omega)
$$ 
for any sequence $E_n$, $s_n$ such that $\chi_{E_n}\rightarrow\chi_E$ in $L^1_{\text{loc}}(\R^N)$ and 
$s_n\rightarrow 1^-$, while \eqref{eq:gammalimsup} means that for every measurable set $E$ and any 
sequence $s_n\rightarrow 1^-$, there exists a sequence $E_n$ with $\chi_{E_n}\rightarrow \chi_E$ in 
$L^1_{\text{loc}}(\R^N)$ such that 
$$
\limsup_{n\to\infty}(1-s_n)P_{s_n}^{\gamma}(E_n;\Omega)\le \frac{\sqrt{2}}{\pi}P^\gamma(E;\Omega).
$$
For an introduction to the $\Gamma$-convergence we refer to \cite{dalmaso}. We notice that the constant in front of the $\Gamma$-limit does not depend on the dimension, as usual in the Gaussian framework.

The paper is structured in the following way. In Section \ref{sec:notprel} we introduce the notation used 
in the paper and state some preliminary results. In particular, in Subsection \ref{ss:fraper} we give 
our definition of fractional Gaussian perimeter and we introduce a fractional Gaussian Sobolev space. In 
Subsections \ref{ss:estoncubs} and \ref{ss:gluing} we state and prove three crucial estimates which 
allow us to prove inequalities \eqref{eq:gammaliminf} and \eqref{eq:gammalimsup}. In Section 
\ref{sec:pfmainth} we prove Theorem \ref{th:maintheorem}; to prove \eqref{eq:gammaliminf} we use 
Lemma \ref{lem:loest} in order to exploit an idea that goes back to \cite{FonMul} and we reduce to 
proving an inequality on Radon-Nikodym derivatives, while for \eqref{eq:gammalimsup} we reduce 
to proving the claim for the energy-dense class of ``transversal'' polyhedra by using 
Lemma \ref{lem:upest}. In Section \ref{sec:convoflocmin}, as in \cite{AmDeMa, DL}, we prove that the 
$\Gamma$-convergence carries out the convergence of local minimisers to a local 
minimiser of the limit functional.

\section{Notation and preliminary results}\label{sec:notprel}
For $N\in\N$ we denote by $\gamma_N$ and $\mathcal{H}^{N-1}_\gamma$ the Gaussian measure on $\R^N$ and the 
$(N-1)$-Hausdorff Gaussian measure
\[
\gamma_N:=\frac{1}{(2\pi)^{N/2}}e^{-\frac{|\cdot|^2}{2}}\Leb{N},
\qquad 
\mathcal{H}^{N-1}_\gamma:=\frac{1}{(2\pi)^{(N-1)/2}}e^{-\frac{|\cdot|^2}{2}}\mathcal{H}^{N-1},
\]
where $\Leb{N}$ and $\mathcal{H}^{N-1}$ are the Lebesgue measure and the Euclidean $(N-1)$-dimensional 
Hausdorff measure, respectively. 
When $k\in \{1,\dots, N\}$, we denote by $\gamma_k$ the standard $k$-dimensional Gaussian measure; when there 
is no ambiguity we simply write $\gamma$ instead of $\gamma_N$ and, with an abuse of notation, we denote by 
$\gamma$ both the measure and its density with respect to $\Leb{N}$. 

The Gaussian perimeter of a measurable set $E$ in an open set $\Omega$ is defined as
$$
P^\gamma(E;\Omega)=\sqrt{2\pi}\sup\left\{\int_E\left(\text{div}\,\varphi-\varphi\cdot x\right)\:d\gamma(x):
\varphi\in C^\infty_c(\Omega;\R^N),\ \|\varphi\|_\infty\le 1\right\}.
$$
Moreover, if $E$ has finite Gaussian perimeter, then $E$ has locally finite Euclidean perimeter and it holds
$$
P^\gamma(E;\Omega)=\mathcal{H}^{N-1}_\gamma({\mathcal F} E\cap\Omega)=\frac{1}{(2\pi)^{\frac{(N-1)}{2}}}
\int_{{\mathcal F} E\cap\Omega}e^{-\frac{|x|^2}{2}}d\mathcal{H}^{N-1}(x),
$$
where ${\mathcal F} E$ is the reduced boundary of $E$. 
If $\Omega=\R^N$, we denote the Gaussian perimeter of $E$ in the whole $\R^N$ simply by $P^\gamma(E)$.
We refer to \cite{AFP} for the properties of 
sets with locally finite perimeter. Let us present an approximation result that will be useful in 
the proof of the $\Gamma-\limsup$ inequality. Its proof is analogous to that of \cite[Proposition 15]{AmDeMa}.

\begin{proposition}\label{pro:transversality}
Let $E\subset\R^N$ a set with $P^\gamma(E;\Omega)<\infty$. Then, for every $\varepsilon>0$, there exists a polyhedral set $\Pi\subset\R^N$ such that
\begin{itemize}
\item[(i)] $\gamma((E\triangle\Pi)\cap\Omega)<\varepsilon$;
\item[(ii)] $|P^\gamma(E;\Omega)-P^\gamma(\Pi;\Omega)|<\varepsilon$;
\item[(iii)] $P^\gamma(\Pi;\partial\Omega)=0$.
\end{itemize}
\end{proposition}

In the sequel, for $\Omega\subset\R^N$ open connected Lipschitz set and for $\delta>0$ we set
\begin{equation}
\begin{split}
\label{eq:omegadeltapiumeno}
&\Omega_\delta^+:=\left\{x\in\Omega^c:d(x,\Omega)<\delta\right\}, \\
&\Omega_\delta^-:=\left\{x\in\Omega:d(x,\Omega^c)<\delta\right\}.
\end{split}
\end{equation}

\subsection{Fractional Sobolev spaces and Fractional perimeters in the Gaussian setting}\label{ss:fraper}

In order to define the fractional perimeter, we introduce the Ornstein-Uhlenbeck 
semigroup, its generator $\Delta_\gamma$, the fractional powers of the generator and the functional setting.

\begin{definition}
	Let $t>0$ and $x\in\R^N$. For $u\in L^1_\gamma(\R^N)$ we define the Ornstein-Uhlenbeck semigroup as
	$$
	e^{t\Delta_\gamma}u(x):=\int_{\R^N}M_t(x,y)u(y)d\gamma(y)
	$$
	where $M_t(x,y)$ denotes the Mehler kernel
	$$
	M_t(x,y):=\frac{1}{(1-e^{-2t})^{N/2}}
	\exp\left(-\frac{e^{-2t}|x|^2-2e^{-t}x\cdot y+e^{-2t}|y|^2}{2(1-e^{-2t})}\right),
	$$
	which satisfies
	$$
	e^{t\Delta_\gamma}1=\int_{\R^N}M_t(x,y)d\gamma(y)=1,
	$$
	for any $t>0$ and any $x\in\R^N$.
	
	The generator of $e^{t\Delta_\gamma}$ acts on sufficiently smooth functions as
	\[
	\Delta_\gamma u = \Delta u - x \cdot Du
	\]
	and is called Ornstein-Uhlenbeck operator; see e.g. \cite{LunMetPal} and the references therein 
	for the main properties of $e^{t\Delta_\gamma}$ and $\Delta_\gamma$.
	
\begin{remark}
	We notice that if we write $x=(x',x_N), y=(y',y_N)\in\R^{N-1}\times\R$, we have 
	$$
	M_t(x,y)=M_t^{N-1}(x',y')M^1_t(x_N,y_N),
	$$
	for any $t>0$, where for $k\in\left\{1,\ldots,N\right\}$, $M^k_t(\cdot,\cdot)$ denotes the Mehler kernel
	 in $\R^k$. When there is no ambiguity we omit the superscript $k$.
\end{remark}

Since $-\Delta_\gamma$ is a positive definite and selfadjoint operator which generates a $C_0$-semigroup 
of contractions in $L^2_\gamma(\R^N)$, we can define its fractional powers by means of spectral 
decomposition via the Bochner subordination formula. 
In particular, for $s\in(0,1)$ and $x\in\R^N$ the fractional Ornstein-Uhlenbeck operator is defined as
	\begin{equation}
	\label{eq:bochnersubordination}
	\begin{split}
	(-\Delta_\gamma)^su(x):&=\frac{1}{\Gamma(-s)}\int_0^{\infty}\frac{e^{t\Delta_\gamma} u(x)-u(x)}
	{t^{s+1}}dt
	\\
	&=\frac{1}{\Gamma(-s)}\int_0^{\infty}\frac{dt}{t^{s+1}}\int_{\R^N}M_t(x,y)(u(y)-u(x))d\gamma(y) \\
	&=\frac{1}{|\Gamma(-s)|}\int_{\R^N}\left(u(x)-u(y)\right)K_{2s}(x,y)d\gamma(y),
	\end{split}
	\end{equation}
	where for $\sigma>0$ we have set 
	\begin{equation}\label{defK_s}
	K_\sigma(x,y):=\int_0^{\infty}\frac{M_t(x,y)}{t^{\frac \sigma2+1}}\, dt,
	\end{equation}
and the right-hand side in \eqref{eq:bochnersubordination} has to be intended in the Cauchy principal 
value sense. 
\end{definition}

The definition of the kernel $K_\sigma$ suggests the following definition of fractional Gaussian 
Sobolev spaces

\begin{definition}
	\label{def:gaussfracsobspace}
	Let $\Omega\subseteq\R^N$ be an open set, $s\in(0,1)$ and $1\le p<\infty$. We define the fractional 
	Gaussian Sobolev space $W^{s,p}_\gamma(\Omega)$ as
	$$
	W^{s,p}_\gamma(\Omega):=\left\{u\in L^p_\gamma(\Omega);\ [u]_{W^{s,p}_\gamma(\Omega)}<\infty\right\},
	$$
	where
	$$
	[u]_{W^{s,p}_\gamma(\Omega)}:=\left(\int_\Omega d\gamma(x)\int_\Omega|u(x)-u(y)|^p
	K_{sp}(x,y)d\gamma(y)\right)^{1/p},
	$$
	and $K_{sp}$ is deined in \eqref{defK_s} with $\sigma=sp$. 
\end{definition}

\begin{remark}
	The integrability of the function 
	$$
	(0,\infty)\ni t\mapsto \frac{M_t(x,y)}{t^{\frac{sp}{2}+1}}
	$$
	near zero, for any $x,y\in\R^N$, $x\neq y$, is ensured by the fact that
	\begin{equation}
	\label{eq:kernelquotient}
	\lim_{t\to 0^+}\frac{M_t(x,y)}{H_t(|x-y|)}=(2\pi)^{N/2}e^{\frac{|x|^2}{4}}e^{\frac{|y|^2}{4}}
	\quad\text{for any}\quad x,y\in\R^N,
	\end{equation}
	where $H_t(\cdot)$ is the Gauss-Weierstrass kernel defined in \eqref{GWkernel}. It is easily 
	seen that the equality in formula \eqref{eq:kernelquotient} it is not true for any $t>0$.
\end{remark}

Now, we make more precise the definition of fractional Gaussian perimeter 
\eqref{eq:fracgaussperimeter} given in Section \ref{sec:intro}.

\begin{definition}
	Let $\Omega\subset\R^N$ be a connected open set with Lipschitz boundary, and $E\subset\R^N$ a 
	measurable set. We define the Gaussian $s$-perimeter of $E$ in $\Omega$ as
	$$
	P^\gamma_s(E;\Omega):=P^{\gamma,L}_s(E;\Omega)+P^{\gamma,NL}_s(E;\Omega),
	$$
	where the {\em local part} is 
	\begin{equation}
	P^{\gamma,L}_s(E;\Omega):=\int_{E\cap\Omega}d\gamma(x)\int_{E^c\cap\Omega}K_s(x,y)d\gamma(y),
	\end{equation}
	and the {\em nonlocal part} is 
	\begin{equation}\label{defNonlocal}
	P^{\gamma,NL}_s(E;\Omega):=\int_{E\cap\Omega}d\gamma(x)\int_{E^c\cap\Omega^c}K_s(x,y)d\gamma(y)+
	\int_{E\cap\Omega^c}d\gamma(x)\int_{E^c\cap\Omega}K_s(x,y)d\gamma(y).
	\end{equation}
	As for the Gaussian perimeter we omit the second argument in $P^\gamma_s(E;\Omega)$ if $\Omega=\R^N$.
\end{definition}

\begin{remark}
	We notice that, since $K_s(y,x)=K_s(x,y)$ for every $s\in(0,1)$ and $x,y\in\R^N$, we have  $P^\gamma_s(E^c;\Omega)=P^\gamma_s(E;\Omega)$.
\end{remark}

As already observed in Section \ref{sec:intro} the definition in \eqref{eq:fracgaussperimeter} is equivalent to the one given in \cite{CaCiLaPa, NovPalSir} thanks to the following integration by parts formula
$$
\frac 12[u]^2_{H^s_\gamma(\R^N)}=\int_{\R^N}u(-\Delta_\gamma)^su\:d\gamma.
$$
Indeed
\begin{equation*}
\begin{split}
\frac 12[u]^2_{H^s_\gamma(\R^N)}=&\frac 12\int_{\R^N}d\gamma(x)\int_{\R^N}|u(x)-u(y)|^2K_{2s}(x,y)d\gamma(y) \\
=&\frac 12\left(\int_{\R^N}u(x)d\gamma(x)\int_{\R^N}(u(x)-u(y))K_{2s}(x,y)d\gamma(y)\right. \\
&-\left.\int_{\R^N}d\gamma(x)\int_{\R^N}u(y)(u(x)-u(y))K_{2s}(x,y)d\gamma(y)\right)\\
=&\int_{\R^N}u(x)\int_{\R^N}(u(x)-u(y))K_{2s}(x,y)d\gamma(y)=\int_{\R^N}u(x)(-\Delta_\gamma)^s u(x) d\gamma(x)
\end{split}
\end{equation*}
where in the third equality we switched $x$ and $y$ and used the symmetry of the kernel. If $u=\chi_E$ for some measurable set $E$ we have
$$
P_s^\gamma(E)=\int_E(-\Delta_\gamma)^{s/2}\chi_E d\gamma.
$$

Another useful inequality which involves \eqref{eq:fracgaussperimeter} is the fractional Gaussian isoperimetric inequality in its analytic form, which reads 
\begin{equation}
\label{eq:gaussisopine}
P_s^\gamma(E) \ge I_s(\gamma(E)),
\end{equation}
where $I_s:(0,1)\rightarrow (0,\infty)$ denotes the fractional Gaussian isoperimetric function, i.e., 
the function that associates to $m\in (0,1)$ the fractional Gaussian perimeter of a halfspace having 
Gaussian measure $m$, and in \eqref{eq:gaussisopine} equality holds if and only if $E$ is a halfspace 
(see \cite{NovPalSir}).

The kernel $K_s$ satisfies the following estimate. 

\begin{lemma}
	For any $x,y\in\R^N$ and for any $s\in(0,1)$ we have
	\begin{equation}
	\label{eq:fractkernellowerbound}
	K_s(x,y)\ge\frac{C_{N,s}}{|x-y|^{N+s}},
	\end{equation}
	where $C_{N,s}:=2^{s+\frac N2}\Gamma\left(\frac{s+N}{2}\right).$
\end{lemma}
\begin{proof}
	For any $x,y\in\R^N$ we have
	\begin{align}\label{eq:heatkernellowerbound}
	M_t(x,y)&=\frac{1}{(1-e^{-2t})^{N/2}}\exp\left(-\frac{e^{-2t}|x|^2-2e^{-t}x\cdot y
		+e^{-2t}|y|^2}{2(1-e^{-2t})}\right)
	\\ \nonumber
	&\ge \frac{1}{(2t)^{N/2}}\exp\left(-\frac{e^{-t}|x-y|^2}{2(1-e^{-2t})}\right) 
	\ge \frac{1}{(2t)^{N/2}}\exp\left(-\frac{|x-y|^2}{4t}\right)=(2\pi)^{N/2}H_t(|x-y|),
	\end{align}
	where in the first inequality we used the fact that $e^{-2t}\le e^{-t}$ and $1-e^{-2t}\le 2t$ for any $t\ge 0$, while in the second we used that $\frac{e^{-t}}{2(1-e^{-2t})}\le\frac{1}{4t}$ for any $t>0$. By dividing both sides of \eqref{eq:heatkernellowerbound} by $t^{\frac s2+1}$ and integrating with respect $t$ in $(0,\infty)$ we get the thesis.
\end{proof} 

\begin{lemma}
	\label{lem:upasest} 
	For any $x,y\in\R^N$ and for any $s\in(0,1)$, the following estimate
	\begin{equation}
	\label{eq:upperadialestimate}
	K_s(x,y)\le e^{\frac{|x|^2}{4}}e^{\frac{|y|^2}{4}}\tilde{K}_s(|x-y|)
	\end{equation}
	holds true, where, for any $z\in\R^N$ we have defined the decreasing kernel
	$$
	\tilde{K}_s(r):=\int_0^{\infty} \exp\left(-\frac{e^tr^2}{2(e^{2t}-1)}\right)
	\frac{dt}{t^{\frac s2+1}(1-e^{-2t})^{N/2}},\qquad r\geq 0.
	$$
	Moreover, for any $a>0$ there exists $R_a>0$ such that for any $s\in(0,1)$ the kernel $\tilde{K}_s$ satisfies the summability condition
	\begin{equation}
	\label{eq:kernelsummability}
	\tilde{K}_s(|x|)\in L^1(B_{R_a},|\cdot|)\cap L^1(B_{R_a}^c,e^{-a|\cdot|^2}).
	\end{equation}
\end{lemma}
\begin{proof}
	The estimate in \eqref{eq:upperadialestimate} simply follows by noticing that
	$$
	M_t(x,y)=\frac{1}{(1-e^{-2t})^{N/2}}
	\exp\left(-\frac{e^t|x-y|^2}{2(e^{2t}-1)}\right)
	\exp\left(\frac{(e^t-1)(|x|^2+|y|^2)}{2(e^{2t}-1)}\right)
	$$
	and
	$$
	\frac{e^t-1}{2(e^{2t}-1)}\le\frac 14
	$$
	for any $t>0$. For every $R>0$ we have 
	\begin{equation}
	\begin{split}
	J_1:&=\int_{B_{R}}|x|dx\int_1^{\infty}\frac{\exp\left(-\frac{e^t|x|^2}{2(e^{2t}-1)}\right)}
	{t^{\frac s2+1}(1-e^{-2t})^{N/2}}dt 
	\\
	&=N\omega_N\int_0^{R}\rho^Nd\rho\int_1^{\infty}\frac{\exp\left(-\frac{e^t\rho^2}{2(e^{2t}-1)}\right)}
	{t^{\frac s2+1}(1-e^{-2t})^{N/2}}dt
	\le N\omega_N\frac{R^{N+1}}{N+1}\frac{2}{s(1-e^{-2})}
	\end{split}
	\end{equation}
	and
	\begin{equation}
	\begin{split}
	J_2:&=\int_{B_{R}}|x|dx\int_0^1\frac{\exp\left(-\frac{e^t|x|^2}{2(e^{2t}-1)}\right)}
	{t^{\frac s2+1}(1-e^{-2t})^{N/2}}dt 
	\\
	&=N\omega_N\int_0^{R}\rho^N d\rho\int_0^1\frac{\exp\left(-\frac{e^t\rho^2}{2(e^{2t}-1)}\right)}
	{t^{\frac s2+1}(1-e^{-2t})^{N/2}}dt 
	\\
	&\le N\omega_N\int_0^1\frac{dt}{t^{\frac s2+1}(1-e^{-2t})^{N/2}}
	\int_0^{\infty}\rho^N\exp\left(-\frac{e^t\rho^2}{2(e^{2t}-1)}\right)d\rho 
	\\
	&=2^{\frac{N}{2}}N\omega_N\Gamma\left(\frac{N+1}{2}\right)
	\int_0^1\frac{(e^{2t}-1)^{1/2}e^{\frac{N-1}{2}t}}{t^{\frac s2+1}}dt 
	\\
	&\le 2^{\frac{N+1}{2}}N\omega_N\Gamma\left(\frac{N+1}{2}\right)e^{\frac{N+1}{2}}
	\int_0^1 t^{\frac{1-s}{2}-1}dt
	\\
	&=2^{\frac{N+3}{2}}N\omega_N\Gamma\left(\frac{N+1}{2}\right)\frac{e^{\frac{N+1}{2}}}{1-s},
	\end{split}
	\end{equation}
	where in the second equality in the right-hand side we performed the change of variable
	$$
	w:=\frac{e^t\rho^2}{2(e^{2t}-1)},
	$$
	and in the second inequality we used that for any $t\in[0,1]$
	$$
	e^{2t}-1\le 2e^2t.
	$$
	Therefore
	$$
	J_1+J_2=\int_{B_{R}}|x|\tilde{K}_s(|x|)dx\le 
	N\omega_N\left(\frac{R^{N+1}}{N+1}\frac{2}{s(1-e^{-2})}
	+2^{\frac{N+3}{2}}\Gamma\left(\frac{N+1}{2}\right)\frac{e^{\frac{N+1}{2}}}{1-s}\right)<\infty,
	$$
	for any $R>0$ and $s\in(0,1)$.
	
	Fix now $a>0$ and let $R>0$. We have
	\begin{equation}
	\label{eq:integrabilityatinfinity}
	\begin{split}
	K_1:&=\int_{B^c_{R}}e^{-a|x|^2}dx\int_{R}^{\infty}\frac{\exp\left(-\frac{e^t|x|^2}{2(e^{2t}-1)}\right)}
	{t^{\frac s2+1}(1-e^{-2t})^{N/2}}dt 
	\\ 
	&=N\omega_N\int_{R}^{\infty}\rho^{N-1}e^{-a\rho^2}d\rho
	\int_{R}^{\infty}\frac{\exp\left(-\frac{e^t\rho^2}{2(e^{2t}-1)}\right)}
	{t^{\frac s2+1}(1-e^{-2t})^{N/2}}dt 
	\\
	&\le \frac{N\omega_N}{(1-e^{-2R})^{N/2}}\int_{R}^{\infty}\frac{dt}{t^{\frac s2+1}}
	\int_0^{\infty}\rho^{N-1}e^{-a\rho^2}d\rho 
	\\
	&=\frac{2\pi^{N/2}}{(1-e^{-2R})^{N/2}}\frac{R^{-\frac s2}}{s\sqrt{a}},
	\end{split}
	\end{equation}
	and
	\begin{equation}
	\label{eq:estimatelocal}
	\begin{split}
	K_2:&=\int_{B^c_{R}}e^{-a|x|^2}dx\int_0^{R}\frac{\exp\left(-\frac{e^t|x|^2}{2(e^{2t}-1)}\right)}
	{t^{\frac s2+1}(1-e^{-2t})^{N/2}}dt 
	\\
	&=N\omega_N\int_{R}^{\infty}\rho^{N-1}e^{-a\rho^2}d\rho
	\int_0^{R}\frac{\exp\left(-\frac{e^t\rho^2}{2(e^{2t}-1)}\right)}
	{t^{\frac s2+1}(1-e^{-2t})^{N/2}}dt 
	\\
	&\le N\omega_N\int_{R}^{\infty}\frac{dt}{t^{\frac s2+1}(1-e^{-2t})^{N/2}}
	\int_0^{\infty}\rho^{N-1}\exp\left(-\frac{e^t\rho^2}{2(e^{2t}-1)}\right)e^{-a\rho^2}d\rho 
	\\
	&\le N\omega_N\int_{R}^{\infty}\frac{dt}{t^{\frac s2+1}(1-e^{-2t})^{N/2}}
	\int_0^{\infty}\rho^{N-1}\exp\left(-\frac{e^t+2a(e^{2t}-1)}{2(e^{2t}-1)}\rho^2\right)d\rho 
	\\
	&=2^{\frac N2-1}N\omega_N\int_{R}^{\infty}\left(\frac{e^{2t}-1}{e^t+2a(e^{2t}-1)}\right)^{N/2}
	\frac{dt}{t^{\frac s2+1}(1-e^{-2t})^{N/2}}\int_0^{\infty}w^{\frac N2-1}e^{-w}dw 
	\\
	&=2^{\frac N2-1}N\omega_N\Gamma\left(\frac N2\right)
	\int_{R}^{\infty}\left(\frac{e^{2t}}{2ae^{2t}+e^t-2a}\right)^{N/2}\frac{dt}{t^{\frac s2+1}}, 
	\end{split}
	\end{equation}
	where in the second equality we performed the change of variable
	$$
	w:=\frac{e^t+2a(e^{2t}-1)}{2(e^{2t}-1)}\rho^2.
	$$
	To conclude, we notice that
	\begin{equation}
	\label{eq:logestimate}
	0<\frac{e^{2t}}{2ae^{2t}+e^t-2a}\le\frac{1}{2a}\quad
	\text{for any}\quad t\ge\log(2a)\quad\text{if}\quad a>\frac 12,
	\end{equation}
	and so in order to have $K_2<\infty$ we can choose $R=R_a:=\log(2a)$. On the other side the 
	inequality in \eqref{eq:logestimate} is true for any $t>0$ if $a\in\left(0,\frac 12\right]$ (and we 
	can choose $R_a:=1$ for simplicity).
	
	Now, putting together \eqref{eq:integrabilityatinfinity}, \eqref{eq:estimatelocal} and \eqref{eq:logestimate}, with this choice of $R_a$ we obtain 
	$$
	K_1+K_2=\int_{B^c_{R_a}}\tilde{K}_s(|x|)e^{-a|x|^2}dx\le
	\frac{4\pi^{N/2}}{s}R_a^{-\frac s2}\left(\frac{1}{2\sqrt{a}(1-e^{-2R_a})^{N/2}}
	+\frac{1}{2a^{N/2}}\right)<\infty 
	$$
	for any $a>0$ and $s\in(0,1)$.
\end{proof}

In the Gaussian framework, in analogy with the Euclidean one, we have a Coarea formula.
\begin{lemma}[Coarea formula]
	\label{lem:coarea}
	For every measurable function $u:\Omega\rightarrow [0,1]$ it holds that
	$$
	\frac 12[u]_{W^{s,1}_\gamma(\Omega)}=\int_0^1 P^{\gamma,L}_s(\{u>t\};\Omega)dt
	$$
\end{lemma}
\begin{proof}
	Given $x,y\in\Omega$, the function $[0,1]\ni t\mapsto\chi_{\{u>t\}}(x)-\chi_{\{u>t\}}(y)$ takes 
	values in $\{-1,0,1\}$ and it is nonzero in the interval $(\min\{u(x),u(y)\},\max\{u(x),u(y)\})$. Therefore
	$$
	|u(x)-u(y)|=\int_0^1|\chi_{\{u>t\}}(x)-\chi_{\{u>t\}}(y)|dt
	$$
	and
	\begin{align*}
	|\chi_{\{u>t\}}(x)-\chi_{\{u>t\}}(y)|&=
	|\chi_{\{u>t\}}(x)-\chi_{\{u>t\}}(y)|^2
	\\
	&=\chi_{\{u>t\}}(x)\chi_{\Omega\setminus\{u>t\}}(y)+\chi_{\{u>t\}}(y)\chi_{\Omega\setminus\{u>t\}}(x).
	\end{align*}	
	Substituting we obtain
	\begin{equation*}
	\begin{split}
	[u]_{W^{s,1}_\gamma(\Omega)}&=\int_\Omega d\gamma(x)\int_\Omega K_s(x,y)d\gamma(y) 
	\int_0^1|\chi_{\{u>t\}}(x)-\chi_{\{u>t\}}(y)|^2dt 
	\\
	&=2\int_0^1 dt\int_{\{u>t\}}d\gamma(x)\int_{\Omega\setminus\{u>t\}}K_s(x,y)d\gamma(y)
	\\
	&=2\int_0^1P^{\gamma,L}_s(\{u>t\};\Omega)dt.
	\end{split}
	\end{equation*}
\end{proof}

\begin{corollary}
	For $E\subset\R^N$ measurable set, $s\in(0,1)$ and $1\le p<\infty$ we have 
	$$
	P^{\gamma,L}_s(E;\Omega)=\frac 12[\chi_E]_{W^{s,1}_\gamma(\Omega)}=
	\frac 12[\chi_E]^p_{W^{\frac sp,p}_\gamma(\Omega)}.
	$$
\end{corollary}
\begin{proof}
	For the first equality it is sufficient to apply Coarea formula \ref{lem:coarea} to $u=\chi_E$. 
	The second equality follows from the same computations by noticing that $K_{\frac spp}(x,y)=K_s(x,y)$ 
	and that $|\chi_E(x)-\chi_E(y)|^p=|\chi_E(x)-\chi_E(y)|^2$ for any $x,y\in\R^N$, $s\in(0,1)$ and 
	$1\le p<\infty$.
	\end{proof}

Now we prove a compactness criterion. This result, combined with the lower semicontinuity of perimeters, ensures existence of local minimisers thanks to direct method of Calculus of Variations. 

\begin{lemma}[Compactness Criterion]
	Let $(E_n)$ be a sequence of measurable sets, let $s_n\to 1^-$ as $n\to\infty$ and
	\begin{equation}
	\sup_{n\in\N}(1-s_n)P^{\gamma,L}_{s_n}(E_n;\Omega')<\infty\quad\forall\Omega'\Subset\Omega.
	\end{equation}
	Then, there exists a subsequence $(E_{n_k})$ and a set $E$ with locally finite perimeter in $\Omega$ 
	such that $\chi_{E_{n_k}}\rightarrow\chi_E$ in $L_{\text{\rm loc}}^1(\Omega)$.
\end{lemma}
\begin{proof}
	We simply notice that, thanks to \eqref{eq:upperadialestimate}, it holds that
	\begin{equation}
	\label{eq:compactnessfrekol}
	(1-s_n)P^{\gamma,L}_{s_n}(E_n;\Omega)\le(1-s_n)P^L_{\tilde{K}_{s_n}}(E_n;\Omega),
	\end{equation}
	where in the right-hand side of \eqref{eq:compactnessfrekol} for any $n\in\N$ the quantity 
	$P^L_{\tilde{K}_{s_n}}(E_n;\Omega)$ denotes the local part of the Euclidean nonlocal perimeter 
	with respect to the radial kernel $\tilde{K}_{s_n}$ of $E_n$ in $\Omega$.
	
	The rest of the proof is a simple consequence of the Fréchet-Kolmogorov compactness criterion in $L^1_{\text{loc}}$ (see for instance \cite[Theorem 3.5]{BerPal}).
\end{proof}

\subsection{Estimates on small cubes}\label{ss:estoncubs}
In this section we prove lower and upper estimates on the integral of the kernel $K_s$ on small 
cubes that are crucial in order to obtain the precise value of the constants in \eqref{eq:gammaliminf} 
and \eqref{eq:gammalimsup}. The upper estimate holds true for every $s\in (0,1)$, the lower estimate 
holds true only in the limit $s\uparrow 1$. 

\begin{lemma}[Estimate from above]\label{lem:upest}
Let $N\ge 2$, let $\Sigma$ be a $(N-1)$-dimensional plane and $x_0\in \Sigma$ and denote $\Sigma^\pm$ 
the two halfspaces determined by $\Sigma$. Let $Q_r(x_0)$ be a cube centred in $x_0$ with side length 
$r$ and faces either parallel or orthogonal to $\Sigma$ and set
$$
Q^\pm_r(x_0):=\Sigma^\pm\cap Q_r(x_0).
$$
Then, there is $C>0$ such that for any $s\in(0,1)$ the estimate 
\begin{equation}\label{eq:dom0}
(1-s)\int_{Q^+_r(x_0)}\int_{Q^-_r(x_0)}K_s(x,y)\:d\gamma(x)\:d\gamma(y)\le
\frac{1}{s2^{\frac{N-1-s}{2}}\pi^{\frac{N+1}{2}}} r^{N-1}e^{-\frac{|x_0|^2}{2}}(1+Cr)
\end{equation}
holds.
\end{lemma}
\begin{proof}
Without loss of generality, we can suppose $\Sigma=\left\{x\in\R^N:x_N=(x_0)_N\right\}$ and 
$Q_r(x_0)=\left\{x\in\R^N: \max_{i=1,\ldots,N}|x_i-(x_0)_i|<\frac{r}{2}\right\}$.
In the sequel we write $x=(x',x_N)\in \R^{N-1}\times\R$. We have
\begin{align*}
&\int_{Q^+_r(x_0)}\int_{Q^-_r(x_0)}K_s(x,y)\:d\gamma(x)\:d\gamma(y)\\
&=\int_{(x_0)_N}^{(x_0)_N+\frac r2}d\gamma_{1}(y_N)\int_{(x_0)_N-\frac r2}^{(x_0)_N}d\gamma_{1}(x_N)\int_{Q^{N-1}_r(x'_0)}d\gamma_{N-1}(x')\int_{Q^{N-1}_r(x'_0)}K_s(x,y)\:d\gamma_{N-1}(y'),
\end{align*}
where $Q^{N-1}_r(x'_0)=\left\{x'\in\R^{N-1}:\max_{i=1,\ldots,N-1}|x'_i-(x'_0)_i|<\frac{r}{2} \right\}$.

We estimate the integrand with respect to the $x'$ variable:
\begin{align}\label{eq:dom1}
&\int_{(x_0)_N}^{(x_0)_N+\frac r2}d\gamma_{1}(y_N)\int_{(x_0)_N-\frac r2}^{(x_0)_N}d\gamma_{1}(x_N)\int_{\R^{N-1}}K_s(x,y)\:d\gamma_{N-1}(y')
\\ \nonumber
&=\int_{(x_0)_N}^{(x_0)_N+\frac r2}d\gamma_{1}(y_N)\int_{(x_0)_N-\frac r2}^{(x_0)_N}d\gamma_{1}(x_N)\int_{\R^{N-1}}d\gamma_{N-1}(y')\int_0^{\infty}\frac{M^1_t(x_N,y_N)M^{N-1}_t(x',y')}{t^{\frac s2+1}}dt
\\ \nonumber
&=\int_{(x_0)_N}^{(x_0)_N+\frac r2}d\gamma_{1}(y_N)\int_{(x_0)_N-\frac r2}^{(x_0)_N}d\gamma_{1}(x_N)
\int_0^{\infty}\frac{M^1_t(x_N,y_N)}{t^{\frac s2+1}}dt\int_{\R^{N-1}}M^{N-1}_t(x',y')d\gamma_{N-1}(y')
\\ \nonumber
&=\frac{r^2}{2\pi}\int_0^{\infty}\frac{dt}{t^{\frac{s}{2}+1}(1-e^{-2t})^{1/2}}\int_0^\frac12 dy_N\int_{-\frac12}^0\exp\left(-r^2\frac{|x_N-y_N|^2}{2(1-e^{-2t})}\right) \cdot
\\ \nonumber
&\quad\cdot \exp\left(-\frac{((x_0)_N+rx_N)((x_0)_N+ry_N)}{1+e^{-t}}\right)\:dx_N,
\end{align}
where in the second equality we used Tonelli Theorem and in the third we used that
\begin{equation*}
\int_{\R^{N-1}}M^{N-1}_t(x',y')d\gamma_{N-1}(y')=1,\quad\text{for any}\quad t>0\quad\text{and}\quad x'\in\R^{N-1}
\end{equation*}
and we performed the changes of variables $x_N\to\frac{x_N-(x_0)_N}{r}$ and $y_N\to\frac{y_N-(x_0)_N}{r}$.

Since there exists $C>0$ such that, for $r$ sufficiently small,
\begin{equation}\label{eq:stimaesponente}
-\frac{|(x_0)_N|^2}{(1+e^{-t})}-Cr\le-\frac{((x_0)_N+rx_N)((x_0)_N+ry_N)}{(1+e^{-t})}\le-\frac{|(x_0)_N|^2}{(1+e^{-t})}+Cr
\end{equation}
uniformly in $t>0$ and $|x_N|,|y_N|\le 1$, we can estimate the integrand with respect to $t$ in \eqref{eq:dom1} as follows
\begin{align}\label{eq:dom2}
&\frac{1}{t^{\frac{s}{2}+1}(1-e^{-2t})^{1/2}}
\int_0^\frac12 \! dy_N\int_{-\frac12 }^0\!\exp\Bigl(-r^2\frac{|x_N-y_N|^2}{2(1-e^{-2t})}\Bigr)
\exp\Bigl(-\frac{((x_0)_N+rx_N)((x_0)_N+ry_N)}{1+e^{-t}}\Bigr)dx_N
\nonumber \\  \nonumber
&\le\frac{1}{t^{\frac{s}{2}+1}}\int_0^\frac12 dy_N
\int_{-\frac12 }^0\exp\left(-r^2\frac{|x_N-y_N|^2}{2(1-e^{-2t})}\right)\:dx_N 
\frac{1}{(1-e^{-2t})^{1/2}}\exp\left(-\frac{|(x_0)_N|^2}{1+e^{-t}}\right)(1+Cr)
\\ \nonumber
&\le \frac{1}{t^{\frac{s}{2}+1}}\int_0^{\infty}dy_N
\int_0^{\infty}\exp\left(-r^2\frac{|x_N+y_N|^2}{2(1-e^{-2t})}\right)\:dx_N 
\frac{1}{(1-e^{-2t})^{1/2}}\exp\left(-\frac{|(x_0)_N|^2}{1+e^{-t}}\right)(1+Cr)
\\  \nonumber
&=\left(\int_0^{\pi/2}d\theta\int_0^{\infty}\rho e^{-\rho^2(1+\sin(2\theta))}d\rho\right)\frac{2}{r^2}\frac{(1-e^{-2t})^{1/2}}{t^{\frac{s}{2}+1}}\exp\left(-\frac{|(x_0)_N|^2}{1+e^{-t}}\right)(1+Cr) 
\\  \nonumber 
&=\left(\frac 12\int_0^{\pi/2}\frac{d\theta}{1+\sin(2\theta)}\right)\frac{2}{r^2}\frac{(1-e^{-2t})^{1/2}}{t^{\frac{s}{2}+1}}\exp\left(-\frac{|(x_0)_N|^2}{1+e^{-t}}\right)(1+Cr) 
\\  
&=\frac{1}{r^2}\frac{(1-e^{-2t})^{1/2}}{t^{\frac{s}{2}+1}}\exp\left(-\frac{|(x_0)_N|^2}{1+e^{-t}}\right)(1+Cr),
\end{align}
where we performed the change of variables
\begin{equation}\label{eq:polar}
\begin{cases*}
	\frac{rx_N}{\sqrt{2(1-e^{-2t})}}=\rho\cos \theta 
	\\
	\frac{ry_N}{\sqrt{2(1-e^{-2t})}}=\rho\sin \theta. 
	\\
\end{cases*}
\end{equation}
Putting \eqref{eq:dom2} in \eqref{eq:dom1} and integrating with respect to $\gamma_{N-1}(x')$, we obtain
\begin{equation*}
\begin{split}
&\int_{Q^+_r(x_0)}\int_{Q^-_r(x_0)}K_s(x,y)\:d\gamma(x)\:d\gamma(y)
\\
&\le (1+Cr)\frac{1}{(2\pi)^{\frac{N+1}{2}}}\int_{Q^{N-1}_r(x'_0)}e^{-\frac{|x'|^2}{2}}\:dx'\int_0^{\infty}\frac{(1-e^{-2t})^{1/2}}{t^{\frac{s}{2}+1}}\exp\left(-\frac{|(x_0)_N|^2}{1+e^{-t}}\right)\:dt
\\
&\le (1+Cr)\frac{1}{(2\pi)^{\frac{N+1}{2}}}r^{N-1}\int_{Q^{N-1}_1}e^{-\frac{|x'_0+rx'|^2}{2}}\:dx'
\int_0^{\infty}\frac{(1-e^{-2t})^{1/2}}{t^{\frac{s}{2}+1}}\exp\left(-\frac{|(x_0)_N|^2}{1+e^{-t}}\right)\:dt
\\
&\le(1+Cr)\frac{1}{(2\pi)^{\frac{N+1}{2}}}r^{N-1}e^{-\frac{|x'_0|^2}{2}} \int_{Q^{N-1}_1}\:dx'
\int_0^{\infty}\frac{(1-e^{-2t})^{1/2}}{t^{\frac{s}{2}+1}}\exp\left(-\frac{|(x_0)_N|^2}{1+e^{-t}}\right)\:dt.
\\
\end{split}
\end{equation*}
Using that $\mathcal{H}^{N-1}(Q^{N-1}_1)=1$ we have
\begin{equation}\label{eq:dom3}
\begin{split}
\int_{Q_r^+(x_0)}\int_{Q_r^-(x_0)} K_s(x,y)\:d\gamma(y)\:d\gamma(x)
&\le r^{N-1}(1+Cr)e^{-\frac{|x'_0|^2}{2}} 
\\
&\cdot\frac{1}{(2\pi)^{\frac{N+1}{2}}}\int_0^{\infty}
\frac{(1-e^{-2t})^{1/2}}{t^{\frac{s}{2}+1}}\exp\left(-\frac{|(x_0)_N|^2}{1+e^{-t}}\right)\:dt.
\end{split}
\end{equation}
Let us fix $T>0$ and split the integral on the right hand side in order to estimate separately the integrals on $(0,T)$ and $(T,\infty)$. As
$$
\frac{1}{(1+e^{-t})}\ge\frac{1}{2}
$$
for any $t\ge 0$, we obtain
\begin{equation}\label{eq:split1}
\int_{T}^{\infty}\frac{(1-e^{-2t})^{1/2}}{t^{\frac{s}{2}+1}}
\exp\left(-\frac{|(x_0)_N|^2}{1+e^{-t}}\right)\:dt
\le e^{-\frac{|(x_0)_N|^2}{2}}\int_{T}^{\infty}\frac{dt}{t^{\frac{s}{2}+1}}
=\frac{2}{s}T^{-\frac{s}{2}} e^{-\frac{|(x_0)_N|^2}{2}}.
\end{equation}
For every $t\ge 0$, it holds $(1-e^{-2t})^\frac{1}{2}\le \sqrt{2} t^\frac{1}{2}$
and
$$
\exp\left(-\frac{|(x_0)_N|^2}{1+e^{-t}}\right)\le\exp\left(-\frac{|(x_0)_N|^2}{2}\right),
$$
so we have
\begin{equation}\label{eq:split2}
\begin{split}
\int_0^{T}\frac{(1-e^{-2t})^{1/2}}{t^{\frac{s}{2}+1}}\exp\left(-\frac{|(x_0)_N|^2}{1+e^{-t}}\right)\:dt
&\le \sqrt{2}e^{-\frac{|(x_0)_N|^2}{2}}\int_0^{T}\frac{dt}{t^{\frac{s-1}{2}+1}}
\\
&=\frac{2\sqrt{2}}{1-s}T^{\frac{1-s}{2}}e^{-\frac{|(x_0)_N|^2}{2}}.
\end{split}
\end{equation}
By plugging estimates \eqref{eq:split1} and \eqref{eq:split2} into \eqref{eq:dom3}, multiplying by $(1-s)$ 
and minimising with respect to $T$ (the optimal value for the constant is achieved for $T=1/2$) we get the 
thesis.
\end{proof}

\begin{remark}\label{rem:upest}
Notice that we obtain the same estimate even if we replace $Q^\pm_r(x_0)$ with $\Sigma^\pm$.
\end{remark}

\begin{lemma}[Estimate from below]\label{lem:loest}
Under the hypotheses and notations of Lemma \ref{lem:upest}, there exists $C>0$ such that the
following estimate holds
\begin{equation}\label{eq:loest0}
\liminf_{s\to 1^-}(1-s)\int_{Q^+_r(x_0)}\int_{Q^-_r(x_0)}K_s(x,y)\:d\gamma(x)\:d\gamma(y)\ge\frac{1}{2^{\frac{N-2}{2}}\pi^{\frac{N+1}{2}}}r^{N-1}e^{-\frac{|x_0|^2}{2}}(1-Cr),
\end{equation}
hence
$$
\liminf_{s\to 1^-}\frac{(1-s)}{r^{N-1}}P^{\gamma,L}_s(\Sigma^+;Q_r(x_0))\ge \frac{1}{2^{\frac{N-2}{2}}\pi^{\frac{N+1}{2}}}e^{-\frac{|x_0|^2}{2}}(1-Cr).
$$
\end{lemma}
\begin{proof}
Let $\Sigma$ and $Q_r(x_0)$ be as in Lemma \ref{lem:upest}.
Let us consider $x\in Q^+_r(x_0)$ and estimate
$$
J_s(x):=\frac{e^{-\frac{x^2}{2}}}{(2\pi)^{N/2}}\int_{Q^+_r(x_0)}K_s(x,y)\:d\gamma(y).
$$
It holds
\begin{align}\label{eq:loest1}
J_s(x)=&\int_{Q^+_r(x_0)}\frac{e^{-\frac{x^2}{2}}}{(2\pi)^{N/2}}\:d\gamma(y)
\int_0^{\infty}\frac{M_t(x,y)}{t^{\frac{s}{2}+1}}\:dt
\\ \nonumber
=&\frac{1}{(2\pi)^N}\int_0^{\infty}\frac{dt}{t^{\frac{s}{2}+1}(1-e^{-2t})^{N/2}}
\int_{Q^+_r(x_0)}\exp\left(-\frac{|x-y|^2}{2(1-e^{-2t})}\right)\exp\left(-\frac{x\cdot y}{1+e^{-t}}\right)dy
\\ \nonumber
\ge&\frac{(1-Cr)}{(2\pi)^N}\int_0^{\infty}\exp\left(-\frac{|x_0|^2}{1+e^{-t}}\right)
\frac{dt}{t^{\frac{s}{2}+1}(1-e^{-2t})^{N/2}}
\int_{(x_0)_N}^{(x_0)_N+\frac r2}\exp\left(-\frac{|x_N-y_N|^2}{2(1-e^{-2t})}\right)\:dy_N \cdot
\\ \nonumber
&\qquad \cdot \int_{Q^{N-1}_r(x_0)}\exp\left(-\frac{|x'-y'|^2}{2(1-e^{-2t})}\right)\:dy',
\end{align}
where in the above inequality we applied \eqref{eq:stimaesponente} on each addend of $x\cdot y$.
Now, let us fix $\delta\in(0,1)$. Then, there exists $T_\delta>0$ such that, for any $t\in]0,T_\delta]$
\begin{equation}\label{eq:restopalla}
\gamma_{N-1}\left(Q^{N-1}_{\frac{r}{\sqrt{1-e^{-2t}}}-r}(0)\right)\ge 1-\delta
\end{equation}
and
\begin{equation}
\label{eq:asy}
(1-e^{-2t})^{1/2}\left(1-e^{-\frac{r^2}{8(1-e^{-2t})}}\right)\exp\left(-\frac{|x_0|^2}{1+e^{-t}}\right)
\ge e^{-\frac{|x_0|^2}{2}}(\sqrt{2}t^{1/2}-t)\ge 0
\end{equation}
(indeed, the first factor on the left hand side is $\geq \sqrt{2t}+o(t)$, the second one is 
$\geq 1-t$ and the third is $\geq e^{-\frac{|x_0|^2}{2}}-ct$). 
By \eqref{eq:loest1} we have
\begin{equation*}
\begin{split}
J_s(x)\ge&
\frac{(1-Cr)}{(2\pi)^N}\int_0^{T_\delta}\exp\left(-\frac{|x_0|^2}{1+e^{-t}}\right)\frac{dt}{t^{\frac{s}{2}+1}(1-e^{-2t})^{1/2}}\int_{(x_0)_N}^{(x_0)_N+\frac r2}\exp\left(-\frac{|x_N-y_N|^2}{2(1-e^{-2t})}\right)\:dy_N
\cdot \\
&\quad\cdot\int_{Q^{N-1}_{\frac{r}{\sqrt{1-e^{-2t}}}}(x_0'-x')}\exp\left(-\frac{|z'|^2}{2}\right)\:dz'
\end{split}
\end{equation*}
where we performed the change of variables
$$
z'=\frac{y'-x'}{\sqrt{1-e^{-2t}}}.
$$
Let us notice that the integration domain for $z'$, namely the cube 
$Q^{N-1}_{\frac{r}{\sqrt{1-e^{-2t}}}}(x_0'-x')$, satisfies
$$
Q^{N-1}_{\frac{r}{\sqrt{1-e^{-2t}}}}(x_0'-x')\supset Q^{N-1}_{\frac{r}{\sqrt{1-e^{-2t}}}-r}(0)
$$
for any $x'\in Q^{N-1}_r(x'_0)$. Indeed, if $z'\in Q^{N-1}_{\frac{r}{\sqrt{1-e^{-2t}}}-r}(0)$,
then $|x_i'-(x_0')_i|<\frac{r}{2}$ for $i=1,\ldots,N-1$ and 
$$
|z'_i-(x'-x_0')_i|\le|z'_i|+|x'_i-(x_0')_i|<\frac{r}{2\sqrt{1-e^{-2t}}}-\frac{r}{2}+\frac{r}{2}
=\frac{r}{2\sqrt{1-e^{-2t}}}.
$$
By using \eqref{eq:restopalla} we obtain
\begin{align*}
&J_s(x)\ge\frac{(1-Cr)}{(2\pi)^{\frac{N+1}{2}}}
\int_0^{T_\delta}\exp\left(-\frac{|x_0|^2}{1+e^{-t}}\right)\frac{dt}{t^{\frac{s}{2}+1}(1-e^{-2t})^{1/2}}
\int_{(x_0)_N}^{(x_0)_N+\frac r2}\exp\left(-\frac{|x_N-y_N|^2}{2(1-e^{-2t})}\right)dy_N\cdot
\\
&\qquad\qquad\qquad \cdot\gamma_{N-1}\left(Q^{N-1}_{\frac{r}{\sqrt{1-e^{-2t}}}-r}(0)\right)
\\
&\ge\frac{(1-\delta)(1-Cr)}{(2\pi)^{\frac{N+1}{2}}}
\int_0^{T_\delta}\!\!\!\exp\left(-\frac{|x_0|^2}{1+e^{-t}}\right)\frac{dt}{t^{\frac{s}{2}+1}(1-e^{-2t})^{1/2}}
\int_{(x_0)_N}^{(x_0)_N+\frac r2}\!\!\!\!\exp\left(-\frac{|x_N-y_N|^2}{2(1-e^{-2t})}\right)dy_N.
\end{align*}
Now, we integrate with respect to the $x$ variable
\begin{equation*}
\begin{split}
&\int_{Q^+_r(x_0)}\int_{Q^-_r(x_0)}K_s(x,y)\:d\gamma(x)\:d\gamma(y)=\int_{Q^-_r(x_0)}J_s(x)\:dx
\\
&\ge\frac{(1-\delta)(1-Cr)}{(2\pi)^{\frac{N+1}{2}}}\int_0^{T_\delta}\exp\left(-\frac{|x_0|^2}{1+e^{-t}}\right)\frac{dt}{t^{\frac{s}{2}+1}(1-e^{-2t})^{1/2}} \cdot
\\
&\quad \cdot \int_{(x_0)_N}^{(x_0)_N+\frac r2}\:dy_N\int_{(x_0)_N-\frac r2}^{(x_0)_N}
\exp\left(-\frac{|x_N-y_N|^2}{2(1-e^{-2t})}\right)\:dx_N\cdot\mathcal{L}^{N-1}(Q^{N-1}_r(x'_0))
\\
&=r^{N-1}\frac{(1-\delta)(1-Cr)}{(2\pi)^{\frac{N+1}{2}}}r^2\int_0^{T_\delta}
\exp\left(-\frac{|x_0|^2}{1+e^{-t}}\right)\frac{dt}{t^{\frac{s}{2}+1}(1-e^{-2t})^{1/2}}\cdot
\\
&\quad \cdot \int_0^{\frac 12}dy_N\int_{-\frac 12}^0\exp\left(-r^2\frac{|x_N-y_N|^2}{2(1-e^{-2t})}\right)\:dx_N,
\end{split}
\end{equation*}
where we replaced $x_N\to\frac{x_N-(x_0)_N}{r}$ and $y_N\to\frac{y_N-(x_0)_N}{r}$. Proceeding as in \eqref{eq:dom2}, we can estimate from below the integrand with respect to $t$ in \eqref{eq:loest1} as follows
\begin{align}\label{eq:loest2}
&\exp\left(-\frac{|x_0|^2}{1+e^{-t}}\right)\frac{1}{t^{\frac{s}{2}+1}(1-e^{-2t})^{1/2}}
\int_0^{\frac 12}dy_N\int_{-\frac 12}^0\exp\left(-r^2\frac{|x_N-y_N|^2}{2(1-e^{-2t})}\right)\:dx_N
\\ \nonumber
&\ge \exp\left(-\frac{|x_0|^2}{1+e^{-t}}\right)\frac{1}{t^{\frac{s}{2}+1}}
\int_0^{\frac 12}dy_N\int_0^{\frac 12}\exp\left(-r^2\frac{|x_N+y_N|^2}{2(1-e^{-2t})}\right)\:dx_N
\cdot\frac{1}{(1-e^{-2t})^{1/2}}(1-Cr)
\\ \nonumber
&\ge\exp\left(-\frac{|x_0|^2}{1+e^{-t}}\right)\left(\int_0^{\pi/2}d\theta
\int_0^{\frac{r}{2\sqrt{2(1-e^{-2t})}}}\rho e^{-\rho^2(1+\sin(2\theta))}d\rho\right)
\frac{2}{r^2}\frac{(1-e^{-2t})^{1/2}}{t^{\frac{s}{2}+1}}(1-Cr) 
\\ \nonumber
&\ge\exp\left(-\frac{|x_0|^2}{1+e^{-t}}\right)\left(\frac 12\int_0^{\pi/2}
\frac{d\theta}{1+\sin(2\theta)}\right)\left(1-e^{-\frac{r^2}{8(1-e^{-2t})}}\right)
\frac{2}{r^2}\frac{(1-e^{-2t})^{1/2}}{t^{\frac{s}{2}+1}}(1-Cr) 
\\ \nonumber
&=\frac{1}{r^2}\frac{(1-e^{-2t})^{1/2}}{t^{\frac{s}{2}+1}}\left(1-e^{-\frac{r^2}{8(1-e^{-2t})}}\right)
\exp\left(-\frac{|x_0|^2}{1+e^{-t}}\right)(1-Cr), 
\end{align}
where we performed the change of variables \eqref{eq:polar}. Then, it holds
\begin{equation}\label{eq:loest3}
\begin{split}
&\int_{Q^+_r(x_0)}\int_{Q^-_r(x_0)}K_s(x,y)\:d\gamma(x)\:d\gamma(y)\\
&\ge r^{N-1}\frac{(1-\delta)(1-Cr)}{(2\pi)^{\frac{N+1}{2}}}\int_0^{T_\delta}\frac{(1-e^{-2t})^{1/2}}{t^{\frac{s}{2}+1}}\left(1-e^{-\frac{r^2}{8(1-e^{-2t})}}\right)\exp\left(-\frac{|x_0|^2}{1+e^{-t}}\right)\:dt.
\end{split}
\end{equation}

In view of \eqref{eq:asy} we have
\begin{equation}\label{eq:loestsplit1}
\begin{split}
&\int_0^{T_\delta}\frac{(1-e^{-2t})^{1/2}}{t^{\frac{s}{2}+1}}\left(1-e^{-\frac{r^2}{8(1-e^{-2t})}}\right)\exp\left(-\frac{|x_0|^2}{1+e^{-t}}\right)\:dt
\\
&\ge \sqrt{2}e^{-\frac{|x_0|^2}{2}}\left(\int_0^{T_\delta}\frac{dt}{t^{\frac{s-1}{2}+1}}-\int_0^{T_\delta}\frac{dt}{t^{\frac{s}{2}}}\right)=2\sqrt{2}e^{-\frac{|x_0|^2}{2}}\left(\frac{T_\delta^{\frac{1-s}{2}}}{1-s}-\frac{T_\delta^{\frac{2-s}{2}}}{2-s}\right).
\end{split}
\end{equation}
By plugging estimate \eqref{eq:loestsplit1} into \eqref{eq:loest3} and multiplying by $(1-s)$ we get
$$
(1-s)\int_{Q^+_r(x_0)}\int_{Q^-_r(x_0)}\!\!K_s(x,y)\:d\gamma(x)\:d\gamma(y)\ge
\frac{(1-\delta)(1-Cr)}{2^{\frac{N-2}{2}}\pi^{\frac{N+1}{2}}}r^{N-1}e^{-\frac{|x_0|^2}{2}}
\left(T_\delta^{\frac{1-s}{2}}-\frac{1-s}{2-s}T_\delta^{\frac{2-s}{2}}\right)
$$
and letting $s\to1^-$ we obtain
$$
\liminf_{s\to 1^-}(1-s)\int_{Q^+_r(x_0)}\int_{Q^-_r(x_0)}K_s(x,y)\:d\gamma(x)\:d\gamma(y)\ge\frac{1}{2^{\frac{N-2}{2}}\pi^{\frac{N+1}{2}}}r^{N-1}e^{-\frac{|x_0|^2}{2}}(1-Cr)(1-\delta).
$$
Since $\delta$ is arbitrary, we get the thesis.
\end{proof}

\begin{corollary}
Under the hypotheses of Lemmas \ref{lem:upest} and \ref{lem:loest}, it holds
\begin{align*}
&\lim_{r\to 0^+}\frac{1}{r^{N-1}}
\liminf_{s\to 1^-}(1-s)\int_{Q^+_r(x_0)}\int_{Q^-_r(x_0)}K_s(x,y)\:d\gamma(x)\:d\gamma(y)
\\
&=\lim_{r\to 0^+}\frac{1}{r^{N-1}}
\limsup_{s\to 1^-}(1-s)\int_{Q^+_r(x_0)}\int_{Q^-_r(x_0)}K_s(x,y)\:d\gamma(x)\:d\gamma(y)
=\frac{1}{2^{\frac{N-2}{2}}\pi^{\frac{N+1}{2}}}e^{-\frac{|x_0|^2}{2}}.
\end{align*}
\begin{proof}
It is sufficient to notice that the constant in Lemma \ref{lem:upest} converges to 
$\frac{1}{2^{\frac{N-2}{2}}\pi^{\frac{N+1}{2}}}$ as $s\to1^{-}$.
\end{proof}
\end{corollary}

\subsection{Gluing}\label{ss:gluing}
In this subsection we perform a construction similar to the one in \cite{AmDeMa} that is going to be 
essential to prove the $\Gamma-\liminf$ result. The sets $\Omega_\delta^\pm$ are defined in 
\eqref{eq:omegadeltapiumeno}. To do this we introduce the finite measure 
\begin{equation}\label{deflambda}
\lambda:=\frac{1}{(2\pi)^{N/2}}e^{-\frac{|\cdot|^2}{4}}\mathcal{L}^N.
\end{equation}

\begin{proposition}
	\label{prop:gluing}
	Let $Q\Subset\R^N$ be a Lipschitz set. Given $s\in(0,1)$, $E_1,E_2\subset\R^N$ measurable sets such 
	that $P^{\gamma,L}_s(E_i;Q)<\infty$, $i=1,2$, and given $\delta_1>\delta_2>0$ there is a measurable 
	set $F$ such that
	\begin{enumerate}
		\item $\left\|\chi_F-\chi_{E_1}\right\|_{L^1_\gamma(Q)}\le
		\left\|\chi_{E_1}-\chi_{E_2}\right\|_{L^1_\gamma(Q)},$
		\item $F\cap\left(Q\setminus Q^-_{\delta_1}\right)=E_1\cap\left(Q\setminus Q^-_{\delta_1}\right)$, 
		\quad $F\cap\Omega^-_{\delta_2}=E_2\cap\Omega^-_{\delta_2}.$
		\item For all $\varepsilon>0$ we have 
		\begin{equation*}
		\begin{split}
		P^{\gamma,L}_s(F;Q)\le& P^{\gamma,L}_s(E_1;Q)
		+P^{\gamma,L}_s(E_2;Q^-_{\delta_1+\varepsilon})+2^{N+1}\tilde{K}_s(\varepsilon) 
		\\
		&+C'(N,Q,\delta_1,\delta_2)\left(\frac{c_N}{1-s}+\frac{c'(Q,N)}{s}\right)
		\left\|\chi_{E_1}-\chi_{E_2}\right\|_{L^1_\lambda(Q^-_{\delta_1}\setminus Q^-_{\delta_2})} 
		\\
		&+C'(N,Q,\delta_1,\delta_2)\left(\frac{d_N}{2-s}+\frac{d'(Q,N)}{s}\right)
		\left\|\chi_{E_1}-\chi_{E_2}\right\|_{L^1_\lambda(Q)}.
		\end{split}
		\end{equation*}
	\end{enumerate}
\end{proposition}
\begin{proof}
	Let $\varphi\in C^{\infty}(\R^N)$ such that $0\le\varphi\le 1$ in $Q$, 
	$\varphi\equiv 0$ in $Q^-_{\delta_2}$, $\varphi\equiv 1$ in $Q\setminus Q^-_{\delta_1}$ and 
	$|\nabla\varphi|\le\frac{2}{\delta_1-\delta_2}$.
	Given $u_1,u_2$ two measurable functions, $u,v:Q\rightarrow[0,1]$ s.t. 
	$[u_i]_{W^{s,1}_\gamma(Q)}<\infty$, $i=1,2$.
	Define $w:=\varphi u_1+(1-\varphi)u_2$. For $x,y\in Q$ we have
	\begin{equation*}
	\begin{split}
	w(x)-w(y)=&(\varphi(x)-\varphi(y))u_1(y)+\varphi(x)(u_1(x)-u_1(y)) 
	\\
	&+(1-\varphi(x))(u_2(x)-u_2(y))-u_2(y)(\varphi(x)-\varphi(y)) 
	\\
	=&(\varphi(x)-\varphi(y))(u_1(y)-u_2(y))+\varphi(x)(u_1(x)-u_1(y))
	\\
	&+(1-\varphi(x))(u_2(x)-u_2(y)),
	\end{split}
	\end{equation*}
	and this implies
	$$
	|w(x)-w(y)|\le|\varphi(x)-\varphi(y)||u_1(y)-u_2(y)|+\chi_{\{\varphi\not\equiv 0\}}(x)|u_1(x)-u_1(y)|
	+\chi_{\{\varphi\not\equiv 1\}}(x)|u_2(x)-u_2(y)|.
	$$
	Since $\{\varphi\not\equiv 0\}\subset Q\setminus Q^-_{\delta_2}$ and 
	$\{\varphi\not\equiv 1\}\subset Q^-_{\delta_1}$ we get
	\begin{equation*}
	\begin{split}
	[w]_{W^{s,1}_\gamma(Q)}\le&\int_Q|u_1(y)-u_2(y)|d\gamma(y)
	\int_Q|\varphi(x)-\varphi(y)|K_s(x,y)d\gamma(x) 
	\\
	&+\int_{Q\setminus Q^-_{\delta_2}}d\gamma(y)\int_Q|u_1(x)-u_1(y)|K_s(x,y)d\gamma(x)
	\\
	&+\int_{Q^-_{\delta_1}}d\gamma(y)\int_Q|u_2(x)-u_2(y)|K_s(x,y)d\gamma(x)
	\\
	=:&I_1+I_2+I_3.
	\end{split}
	\end{equation*}
Let us estimate $I_1$. Using \eqref{eq:upperadialestimate} and
	$$
	|\varphi(x)-\varphi(y)|\le|D\varphi(y)||x-y|+\frac 12\left\|D^2\varphi\right\|_\infty|x-y|^2,
	$$
	for $\alpha\in\{1,2\}$ we define the constants $C(N,Q,\alpha,s)$ through the following estimate
	\begin{equation*}
	\begin{split}
	\int_Q|x-y|^\alpha K_s(x,y)d\gamma(x)&\le e^{\frac{|y|^2}{4}}
	\int_Q|x-y|^\alpha\tilde{K}_s(|x-y|)d\lambda(x) \\
	&\le e^{\frac{|y|^2}{4}}\int_{B_{R_Q}(y)}|x-y|^\alpha\tilde{K}_s(|x-y|)d\lambda(x) 
	\\
	&\le \frac{e^{\frac{|y|^2}{4}}}{(2\pi)^{N/2}}\int_{B_{R_Q}(y)}|x-y|^\alpha\tilde{K}_s(|x-y|)dx 
	\\
	&=\frac{N\omega_N}{(2\pi)^{N/2}} e^{\frac{|y|^2}{4}}\int_0^{R_Q}\rho^{N+\alpha-1}d\rho
	\int_0^{\infty}\frac{\exp\left(-\frac{e^t\rho^2}{2(e^{2t}-1)}\right)}{t^{\frac s2+1}(1-e^{-2t})^{N/2}}dt 
	\\
	&=:e^{\frac{|y|^2}{4}}C(N,Q,\alpha,s),
	\end{split}
	\end{equation*}
where $R_Q$ is large enough to get $Q\subset B_{R_Q}(y)$. We have 
	\begin{equation}
	\label{eq:iuno}
	\begin{split}
	I_1\le&\int_Q|u_1(y)-u_2(y)|d\gamma(y)\int_Q\left(|D\varphi(y)||x-y|
	+\frac{\left\|D^2\varphi\right\|_\infty}{2}|x-y|^2\right)K_s(x,y)d\gamma(x)
	\\
	\le& \int_Q|u_1(y)-u_2(y)|d\lambda(y)\int_Q\left(|D\varphi(y)||x-y|
	+\frac{\left\|D^2\varphi\right\|_\infty}{2}|x-y|^2\right)\tilde{K}_s(|x-y|)d\lambda(x)
	\\
	\le& C(Q,\delta_1,\delta_2)\left(C(N,Q,1,s)
	\left\|u_1-u_2\right\|_{L^1_\lambda(Q^-_{\delta_1}\setminus Q^-_{\delta_2})}
	+\frac{C(N,Q,2,s)}{2}\left\|u_1-u_2\right\|_{L^1_\lambda(Q)}\right) 
	\\
	\le& C'(N,Q,\delta_1,\delta_2)\left(\left(\frac{c_N}{1-s}+\frac{c'(Q,N)}{s}\right)
	\left\|u_1-u_2\right\|_{L^1_\lambda(Q^-_{\delta_1}\setminus Q^-_{\delta_2})}\right.
	\\
	&+\left.\left(\frac{d_N}{2-s}+\frac{d'(Q,N)}{s}\right)
	\left\|u_1-u_2\right\|_{L^1_\lambda(Q)}\right),
	\end{split}
	\end{equation}
	where
	$$
c_N:=2^{\frac{N+3}{2}}\Gamma\left(\frac{N+1}{2}\right)e^{\frac{N+1}{2}}, 
c'(Q,N):=\frac{2R_Q^{N+1}}{(N+1)(1-e^{-2})},
d_N:=2^{\frac{N+3}{2}}\Gamma\left(\frac{N+2}{2}\right)e^{\frac{N+2}{2}},
$$
$$
d'(Q,N):=\frac{R_Q^{N+2}}{(N+2)(1-e^{-2})}.
$$
Before going on, let us check that the product $(1-s)C(N,Q,\alpha,s)$ is bounded for $\alpha=1,2$ and 
$s\to1^-$. 
For, if we split $C(N,Q,\alpha,s)$ in the contribution of $(\rho,t)\in(0,R_Q)\times(0,1)$ and 
$(\rho,t)\in(0,R_Q)\times(1,\infty)$ with analogous computions as in Lemma \ref{lem:upasest} we have
\begin{equation}
\label{eq:stimagrande}
\int_0^{R_Q}\rho^{N+\alpha-1}d\rho\int_1^{\infty}\frac{\exp\left(-\frac{e^t\rho^2}{2(e^{2t}-1)}\right)}
{t^{\frac s2+1}(1-e^{-2t})^{N/2}}dt\le\frac{R_Q^{N+\alpha}}{N+\alpha}\frac{2}{s(1-e^{-2})}
\end{equation}
and
\begin{align}
\label{eq:stimapiccola}
\int_0^{R_Q}\rho^{N+\alpha-1}d\rho&\int_0^1\exp\Bigl(-\frac{e^t\rho^2}{2(e^{2t}-1)}\Bigr)
\frac{dt}{t^{\frac s2+1}(1-e^{-2t})^{N/2}} 
\\ \nonumber
&\leq\int_0^1\frac{dt}{t^{\frac s2+1}(1-e^{-2t})^{N/2}}
\int_0^{\infty}\rho^{N+\alpha-1}\exp\Bigl(-\frac{e^t\rho^2}{2(e^{2t}-1)}\Bigr)d\rho 
\\ \nonumber
&=2^{\frac{N+\alpha-1}{2}}\Gamma\left(\frac{N+\alpha}{2}\right)
\int_0^1\frac{(e^{2t}-1)^{\alpha/2}e^{\frac{N-\alpha}{2}t}}{t^{\frac s2+1}}dt 
\\ \nonumber
&\le 2^{\frac{N+2\alpha-1}{2}}\Gamma\left(\frac{N+\alpha}{2}\right)e^{\frac{N+\alpha}{2}}\int_0^1 t^{\frac{\alpha-s}{2}-1}dt
\\ \nonumber
&=2^{\frac{N+2\alpha+1}{2}}\Gamma\left(\frac{N+\alpha}{2}\right)\frac{e^{\frac{N+\alpha}{2}}}{\alpha-s}.
\end{align}
	For $I_2$ we notice that trivially 
	\begin{equation}
	\label{eq:idue}
	I_2\le[u_1]_{W^{s,1}_\gamma(Q)}.
	\end{equation}
	For $I_3$ we have
	\begin{equation}
	\label{eq:itre}
	\begin{split}
	I_3=&\int_{Q^-_{\delta_1}}d\gamma(y)\int_{Q^-_{\delta_1+\varepsilon}}
	|u_2(x)-u_2(y)|K_s(x,y)d\gamma(x)
	\\
	&+\int_{Q^-_{\delta_1}}d\gamma(y)\int_{Q\setminus Q^-_{\delta_1+\varepsilon}}
	|u_2(x)-u_2(y)|K_s(x,y)d\gamma(x) 
	\\
	\le&[u_2]_{W^{s,1}_\gamma(Q^-_{\delta_1+\varepsilon})}+
	2\tilde{K}_s(\varepsilon)\lambda(Q^-_{\delta_1})\lambda(Q\setminus Q^-_{\delta_1+\varepsilon})\le[u_2]_{W^{s,1}_\gamma(Q^-_{\delta_1+\varepsilon})}+2^{N+1}\tilde{K}_s(\varepsilon).
	\end{split}
	\end{equation}
	Summing up \eqref{eq:iuno}, \eqref{eq:idue} and \eqref{eq:itre} we prove (3). Using Lemma \ref{lem:coarea} we deduce that there exists $t^\star\in(0,1)$ such that $F:=\{w>t^\star\}$ and
	$$
	2P^{\gamma,L}_s(F;Q)\le [w]_{W^{s,1}_\gamma(Q)}.
	$$
	If we specialise the previous estimates choosing $u_1:=\chi_{E_1}$ and $u_2:=\chi_{E_2}$ we obtain the desired estimate for the local part of the perimeter. Moreover, by construction the set $F$ satisfies points (1) and (2).
\end{proof}

\section{Proof of the Main Theorem}\label{sec:pfmainth}
\begin{proof}{\textbf{Liminf inequality}}
	Let us prove that $E$ is a Caccioppoli set. If $\Omega'\Subset\Omega$, there is $c_0=c_0(\Omega')$ such
	that $\gamma(x)\gamma(y)\ge c_0$ for every $x,y\in\Omega'$. Therefore, we may compare 
	$P^{\gamma,L}_{s_n}$ with its Euclidean counterpart $P^{L}_{s_n}$ using \eqref{eq:fractkernellowerbound}
	and we get 
	$$
	C_{N,s}c_0\limsup_{n\to\infty}(1-s_n)P^L_{s_n}(E_n;\Omega')\le
	\lim_{n\to\infty}(1-s_n)P^{\gamma,L}_{s_n}(E_n;\Omega')<\infty.
		$$		
By \cite[Theorem 1]{AmDeMa}, we know that $E$ has locally finite perimeter in $\Omega$.  
	Let us denote by $\mathcal{C}$ the family of all $N$-cubes in $\R^N$
	$$
	\mathcal{C}:=\left\{R(x+rQ):\quad x\in\R^N, r>0, R\in SO(N) \right\},
	$$
	where $Q:=\left(-\frac 12,\frac 12\right)^N$ and let $s_n$, $E_n$ be such that $s_n\rightarrow 1$ and 
	$\chi_{E_n}\rightarrow\chi_E$ in $L^1_{\text{loc}}(\R^N)$.
Our claim is
		$$
		\liminf_{n\to\infty}(1-s_n)P^{\gamma,L}_{s_n}(E_n;\Omega)\ge \frac{\sqrt 2}{\pi}P^\gamma(E;\Omega).
		$$
Denote by $\mu$ the perimeter measure $\mu(A):=|D\chi_E|(A)$ for any Borel set $A\subset\Omega$, and notice
that for any $x_0\in{\mathcal F} E$ there exists a rotation $R_{x_0}\in SO(N)$ such that the blow-up 
$\frac{E-x_0}{r}$ locally converges in measure to $R_{x_0}H$ as $r\to 0^+$ and
		\begin{equation}
		\label{eq:degiorgi}
		\lim_{r\to 0^+}\frac{\mu(x_0+rR_{x_0}Q)}{r^{N-1}}=1.
		\end{equation}
		Now, for $C\in\mathcal{C}, C\subseteq\Omega$ we set
		$$
		\alpha_n(C):=(1-s_n)P^{\gamma,L}_{s_n}(E_n;C)\quad\text{and}\quad\alpha(C):=
		\liminf_{n\to\infty}\alpha_n(C).
		$$
		We set $C_r(x_0):=x_0+rR_{x_0}Q$ and define the measure
		$$
		\nu(F):=\int_F\gamma(x)d\mu(x),\quad\forall F\quad\text{Borel set}.
		$$
		We claim that for $\mu$-a.e. $x_0\in\R^N$ it holds
		\begin{equation}
		\label{eq:claimliminf}
		\frac{\sqrt 2}{\pi}\le\liminf_{r\to 0^+}\frac{\alpha(C_r(x_0))}{\nu(C_r(x_0))}.
		\end{equation}
		Indeed, if \eqref{eq:claimliminf} is true, the family
		$$
		\mathcal{A}_\eps:=\left\{C_r(x_0)\subset\Omega:\frac{\sqrt 2}{\pi} \nu(C_r(x_0))\le(1+\eps)\alpha(C_r(x_0))\right\}
		$$
		is a fine covering of $\mu$-almost all of $\Omega$ and using a suitable variant of Vitali's 
		covering Theorem as done in \cite{DL} we get
		$$
		\frac{\sqrt 2}{\pi}P^\gamma(E;\Omega)\le(1+\eps)
		\liminf_{n\to\infty}(1-s_n)P^{\gamma,L}_{s_n}(E_n;\Omega).
		$$
Notice that in the right-hand side of \eqref{eq:claimliminf} we have the Radon-Nikodym derivative of 
$\alpha$ with respect to $\nu$.

Since $\eps>0$ is arbitrary, the $\Gamma-\liminf$ inequality follows.
Therefore we reduce ourselves to proving \eqref{eq:claimliminf}. For $x_0\in{\mathcal F} E$, because of the continuity of the density we have 
		\begin{equation}
		\label{eq:average}
		\lim_{r\to 0^+}\fint_{C_r(x_0)}\gamma(x)d\mu(x)=\gamma(x_0).
		\end{equation}
		Then, it suffices to show that
		$$
		\liminf_{r\to 0^+}\frac{\alpha(C_{r}(x_0))}{r^{N-1}}\ge \frac{\sqrt 2}{\pi}\gamma(x_0).
		$$
		From now on, since $x_0\in {\mathcal F} E$ is arbitrary we assume that $R_{x_0}=I$, so $C_r(x_0)=x_0+rQ$. Let us choose a sequence $r_k$ of radii $r_k\rightarrow 0$ such that
		$$
		\liminf_{r\to 0^+}\frac{\alpha(C_r(x_0))}{r^{N-1}}=\lim_{k\to\infty}\frac{\alpha(C_{r_k}(x_0))}{r_k^{N-1}}.
		$$
		For $k>0$ we choose $i(k)$ large enough such that the following conditions hold
		$$
		\begin{cases*}
			\alpha_{i(k)}(C_{r_k}(x_0))\le \alpha(C_{r_k}(x_0))+r_k^N \\
		1-s_{i(k)}\leq r_k^{N} \\
			\fint_{C_{r_k}(x_0)}|\chi_{E_{i(k)}}-\chi_E|dx<\frac 1k.
		\end{cases*}
		$$
Hence we have
\[
\begin{split}
    \alpha (C_{r_k}(x_0))\geq \alpha_{i(k)}(C_{r_k}(x_0)) -r_k^{N}
    =(1-s_{i(k)})P^{\gamma,L}_{s_{i(k)}}(E_{i(k)};x_0+ r_k Q)-r_k^N.
\end{split}
\] 
Let us fix $\delta>0$. Recalling that the halfspace $H$ passes through the origin, (hence $\beta H=H$ for any $\beta>0$), and using Proposition \ref{prop:gluing} with $F_{i(k)}=E_{i(k)}$ in $x_0 +r_k(Q\setminus Q^-_{\delta})$, 
$F_{i(k)}=x_0+H$ in $x_0+ r_k Q^-_{\delta/2}$ and $\delta_1=\delta,\, \delta_2=\delta/2$ and $\eps= \delta/4$, we have that for any $k\in \mathbb N$ there exists a set $F_{i(k)}$ such that
\begin{equation}
\label{eq:lastinequality}
\begin{split}
  P^{\gamma,L}_{s_{i(k)}}(F_{i(k)};x_0+ r_k Q) \leq &
P^{\gamma,L}_{s_{i(k)}}(E_{i(k)};x_0+ r_k Q)+P^{\gamma,L}_{s_{i(k)}} (x_0+H;x_0+ r_k (Q^-_{5\delta/4 })) 
\\
&+2^{N+1}\tilde{K}_{s_{i(k)}}\left(\frac{\delta}{4}\right)+C_1( \delta)\|\chi_{E_{i(k)}}-\chi_H\|_{L_\lambda^1(x_0+r_k Q)}\\
& +\frac{C_2(\delta)}{1-s_{i(k)}} 
\|\chi_{E_{i(k)}}-\chi_H\|_{L_\lambda^1(x_0+r_k(Q^-_\delta\setminus Q^-_{\delta/2}))}.
\end{split}
\end{equation}
Multiplying both sides of \eqref{eq:lastinequality} by $1-s_{i(k)}$ we have that for $k$ large enough
\begin{equation} \label{eq:stimadalbasso}
\begin{split}
 (1-s_{i(k)}) P^{\gamma,L}_{s_{i(k)}}(F_{i(k)};x_0+ r_k Q)\leq& 
(1-s_{i(k)})(P^{\gamma,L}_{s_{i(k)}}(E_{i(k)};x_0+ r_k Q))\\
&+(1-s_{i(k)})P^{\gamma,L}_{s_{i(k)}} (x_0+H;x_0+ r_k (Q^-_{5\delta/4 }))\\&+
\frac {C_2(\delta) r_k^{N}}{k}+(1-s_{i(k)})2^{N+1}\tilde{K}_{s_{i(k)}}\left(\frac{\delta}{4}\right).
\end{split}
\end{equation}
With an argument similar to the one in Lemma \ref{lem:upest}
we can prove that
$$
(1-s_{i(k)})P^{\gamma,L}_{s_{i(k)}} (x_0+H;x_0+ r_k (Q^-_{5\delta/4})) \leq C\delta e^{-\frac{|x_0|^2}{2}}r_k^{N-1}.
$$
Let us focus on the left hand side of \eqref{eq:stimadalbasso}.
Using the isoperimetric inequality for the fractional Gaussian perimeter in analytic form \eqref{eq:gaussisopine} we have 
$$
\frac{P^\gamma_{s_{i(k)}}(F_{i(k)})}{I_{s_{i(k)}}(\gamma(F_{i(k)}))}
\geq \frac {P^\gamma_{s_{i(k)}}(x_0+H)}{I_{s_{i(k)}}(\gamma(x_0+H))}.
$$
It is easy to prove that the function $I_s$ is Lipschitz. Therefore, we have
$$
P^\gamma_{s_{i(k)}}(F_{i(k)}) \geq (1-C\gamma (F_{i(k)} \triangle (x_0+H)))P^\gamma_{s_{i(k)}}(x_0+H) \geq (1-Cr_k^N)P^\gamma_{s_{i(k)}}(x_0+H).
$$
Notice that this immediately implies
\[
\begin{split}
&(1-s_{i(k)})(P^{\gamma,L}_{s_{i(k)}}(F_{i(k)}; x_0+r_k Q) -P^{\gamma,L}_{s_{i(k)}}(x_0+H; x_0+r_k Q))
\\
&\geq  (1-s_{i(k)})(P^{\gamma,NL}_{s_{i(k)}}(x_0+H; x_0+r_k Q)-P^{\gamma,NL}_{s_{i(k)}}(F_{i(k)}; x_0+r_k Q)) - r_k.  \end{split}
\]
We can prove that the difference between the nonlocal terms goes to zero, following \cite[Lemma 14]{AmDeMa} 
and using \eqref{eq:upperadialestimate}, while for the other terms (note that we need to divide by 
$r_k^{N-1}$) we also note that it behaves as $r_k^N +\delta r_k^{N-1}$. 
Thus using Lemma \ref{lem:loest} we have
$$
(2\pi)^{-\frac{N-1}{2}}\lim_{k\to\infty} \frac{\alpha (C_{r_k})}{r_k^{N-1}}
\geq (2\pi)^{-\frac{N-1}{2}} \lim_{k\to\infty} 
\frac{(1-s_{i(k)})P^{\gamma,L}_{s_{i(k)}}(x_0+H; x_0+r_k Q)}{r_k^{N-1}}-\delta
\geq \frac{\sqrt 2}{\pi} e^{-\frac{|x_0|^2}{2}}-\delta.
$$
Since $\delta$ is arbitrary we get inequality \eqref{eq:gammaliminf}.
\end{proof}

\begin{proof}{\textbf{Limsup inequality.}}
It is enough to prove the $\Gamma-\limsup$ inequality for a collection $\mathcal{B}$ of sets of finite Gaussian perimeter which is dense in energy, i.e., such that for every set $E$ of finite Gaussian perimeter there exists a sequence $(E_k)\subset\mathcal{B}$ with $\chi_{E_k}\to\chi_E$ in $L^1_{\rm loc}(\R^N)$ as $k\to\infty$ and $\limsup_k P^\gamma(E_k;\Omega)=P^\gamma(E;\Omega)$. Following the ideas in \cite[Section 3.2]{AmDeMa}, we consider as $\mathcal{B}$ the collection of polyhedra $\Pi\subset\R^N$ satisfying $P^\gamma(\Pi;\partial\Omega)=0$ as in item (iii) of Proposition \ref{pro:transversality}.
Notice that the transversality condition $P^\gamma(\Pi;\partial\Omega)=0$ in the definition of $\mathcal{B}$
is equivalent to
\begin{equation}\label{eq:transdelta}
\lim_{\delta\to0^+}P^\gamma(\Pi;\Omega_\delta^+\cup\Omega_\delta^-)=0,
\end{equation}
where $\Omega_\delta^+$ and $\Omega_\delta^-$ are defined in \eqref{eq:omegadeltapiumeno}.
Now, given a polyhedron $\Pi\in\mathcal{B}$ and $\delta>0$, we prove that  
\begin{align*}
&\lim_{s\to 1^-}(1-s)P_s^{\gamma,L}(\Pi;\Omega)=\frac{\sqrt{2}}{\pi} P^\gamma(\Pi;\Omega),
\\
&\lim_{s\to 1^-}(1-s)P_s^{\gamma,NL}(\Pi;\Omega)=\frac{2\sqrt{2}}{\pi} P^\gamma(\Pi;\Omega_\delta^+\cup\Omega_\delta^-).
\end{align*}
Passing to the limit as $\delta\to 0^+$ the transversality condition \eqref{eq:transdelta} provides the required inequality. We divide the proof in two steps.
\\
\textbf{Step 1.} Estimate of $P_s^{\gamma,L}(\Pi;\Omega)$. Let us fix $r > 0$ and set
$$
(\partial\Pi)_{r/2}:=\left\{x\in\R^N:d(x,\partial\Pi)<r/2\right\},
\qquad(\partial\Pi)^-_{r/2}:=(\partial\Pi)_{r/2}\cap\Pi.
$$
We can find $N_r$ disjoint cubes of side $r$, say $Q_r(x_i)$ ($i=1,\ldots,N_r$), such that
\begin{itemize}
\item[(i)] $Q_r(x_i)\subset(\partial\Pi)_{r/2}$, $x_i\in\partial\Pi$ and the faces of 
$Q_r(x_i)$ are either parallel or orthogonal to the face of $\partial\Pi$ where $x_i$ lies;
\item[(ii)] any cube $Q_r(x_i)$ intersects exactly one face of $\partial\Pi$ and its distance by the other faces of $\partial\Pi$ is larger than $r/2$;
\item[(iii)] the $N_r$ intersections $D_r(x_i):=Q_r(x_i)\cap\Sigma$ are $(N-1)$-dimensional cubes that satisfy
$$\mathcal{H}^{N-1}_\gamma\left((\partial\Pi\cap\Omega)\setminus\bigcup_{i=1}^{N_r}D_r(x_i)\right)=P^\gamma(\Pi;\Omega)-\sum_{i=1}^{N_r}\mathcal{H}^{N-1}_\gamma(D_r(x_i))\to 0$$
as $r\to 0^+$.
\end{itemize}

\begin{figure}[h]
\centering
\includegraphics[width=0.6\textwidth]{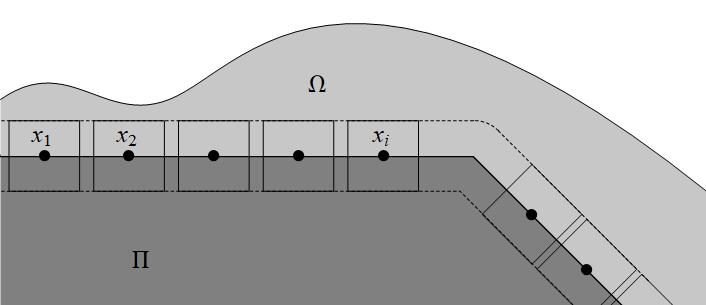}
\caption{A possible collection of cubes satisfying (i), (ii) and (iii).}
\end{figure}
Let us notice that
$$
\left|\mathcal{H}^{N-1}_\gamma(D_r(x_i))-\frac{1}{(2\pi)^{N-1}}r^{N-1}e^{-\frac{|x_i|^2}{2}}
\right| \leq Cr
$$
for some positive constant $C$ independent of $i$.
From now on, for any face $\Sigma$ of $\partial\Pi$, we denote by $\Sigma^+$ and $\Sigma^-$ the two parts of the strip determined by $\Sigma$ lying, respectively, by the side of the outer and the inner normal to $\Sigma$.

We proceed by splitting the integral giving $P_s^{\gamma,L}(\Pi;\Omega)$ in three parts:
\begin{align*}
P_s^{\gamma,L}(\Pi;\Omega)=&\int_{\Pi\cap\Omega}d\gamma(x)\int_{\Pi^c\cap\Omega}K_s(x,y)d\gamma(y)
\\
=&\underbrace{\int_{(\Pi\cap\Omega)\setminus(\partial\Pi)^-_{r/2}}d\gamma(x)\int_{\Pi^c\cap\Omega}K_s(x,y)d\gamma(y)}_{(A)}\\
&+\underbrace{\int_{(\Pi\cap\Omega)\cap\bigcup_{i=1}^{N_r}Q_r(x_i)}d\gamma(x)\int_{\Pi^c\cap\Omega}K_s(x,y)d\gamma(y)}_{(B)}\\
&+\underbrace{\int_{(\Pi\cap\Omega)\cap\left((\partial\Pi)^-_{r/2}\setminus\bigcup_{i=1}^{N_r}Q_r(x_i)\right)}d\gamma(x)\int_{\Pi^c\cap\Omega}K_s(x,y)d\gamma(y)}_{(C)}.
\end{align*}

\begin{figure}[h]
\centering
\includegraphics[width=\textwidth]{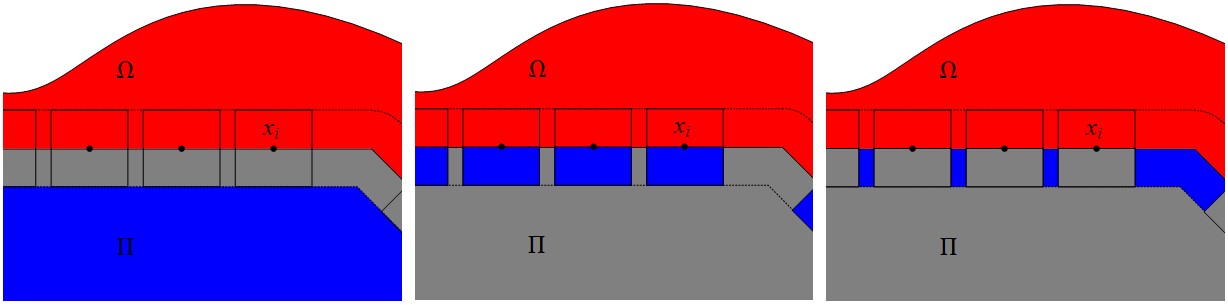}
\caption{Cases (A), (B) and (C). The $x$ variable is in blue, the $y$ variable is in red.}
\end{figure}
\begin{itemize}
\item[(A)] We notice that, for any $x\in(\Pi\cap\Omega)\setminus(\partial\Pi)^-_{r/2}$ and $y\in\Pi^c\cap\Omega$, $|x-y|\ge r/2$. Then, by recalling the upper estimate on the kernel $K_s(x,y)$ in Lemma \ref{lem:upasest} it holds
\begin{align*}
\int_{(\Pi\cap\Omega)\setminus(\partial\Pi)^-_{r/2}}d\gamma(x)&\int_{\Pi^c\cap\Omega}K_s(x,y)d\gamma(y)\\
\le&\int_{(\Pi\cap\Omega)\setminus(\partial\Pi)^-_{r/2}}\:d\lambda(x)\int_{\Pi^c\cap\Omega}\tilde{K}_s\left(\frac r2\right)\:d\lambda(y)\le C_{N,r}<\infty,
\end{align*}
where $\lambda$ is the measure defined in \eqref{deflambda} and the constant $C_{N,r}$ can be estimated as 
the integrals with respect to the variable $t$ in Lemmas \ref{lem:upest} and \ref{lem:loest}.
\item[(B)] Let us estimate separately
$$
\int_{Q_r^-(x_i)}d\gamma(x)\int_{((\Pi^c\cap\Omega)\setminus(\partial\Pi)_{r/2}}K_s(x,y)d\gamma(y)
$$
and
$$
\int_{Q_r^-(x_i)}d\gamma(x)\int_{((\Pi^c\cap\Omega)\cap(\partial\Pi)_{r/2}}K_s(x,y)d\gamma(y).
$$
The first integral can be estimated from above as in case (A) by a positive constant $C_{N,r}$ depending only on $N$ and $r$, since $|x-y|\ge r/2$. 
In order to estimate the second integral, we can observe that, in view of (ii), if $Q_r(x_i)$ intersects 
the face $\Sigma\subset\partial\Pi$, then the contribution of 
$(\Pi^c\cap\Omega\cap(\partial\Pi)_{r/2})\setminus\Sigma^+$ to the integral (A) is again estimated from above by $C_{N,r}$ (the distance between $x$ and $y$ is larger than $r/2$). Then, it remains to estimate
$$
\int_{Q_r^-(x_i)}d\gamma(x)\int_{\Sigma^+\cap(\partial\Pi)_{r/2}}K_s(x,y)d\gamma(y).
$$
Assuming for simplicity that $\Sigma$ lies in a hyperplane $x_N=c$, by repeating the same computations 
as in Lemma \ref{lem:upest}, we obtain 
\begin{align*}
&\int_{Q_r^-(x_i)}d\gamma(x)\int_{\Sigma^+\cap(\partial\Pi)_{r/2}}K_s(x,y)d\gamma(y)\\
&\le r^{N-1}e^{-\frac{|x_i|^2}{2}}(1+Cr)\frac{r^2}{2}
\int_0^{\infty}\frac{1}{(1-e^{-2t})^{1/2}}\exp\left(-\frac{|(x_i)_N|^2}{1+e^{-t}}\right)
\frac{dt}{t^{\frac{s}{2}+1}}\\ 
&\quad\cdot\int_0^{1/2}dy_N\int_{0}^{1/2}\exp\left(-r^2\frac{|x_N+y_N|^2}{2(1-e^{-2t})}\right)\:dx_N\\
&\le\frac{1}{s(1-s)2^{\frac{N-1-s}{2}}\pi^{\frac{N+1}{2}}}r^{N-1}e^{-\frac{|x_i|^2}{2}}(1+Cr).
\end{align*}
\item[(C)] Let us set
$\partial\Pi\cap\Omega=\bigcup_j\Sigma_j$, where $\Sigma_j$ is the intersection of a face of $\partial\Pi$ with $\Omega$. Let us denote by $\pi_j$ the hyperplane containing $\Sigma_j$ and let $\pi_j^-$ and $\pi_j^+$ the halfspaces determined by $\pi_j$ and by the inner and the outer normal vector to $\Pi$, respectively. 
Let us consider the set
$$
(\partial\Pi)^-_{r/2,j}:=\left\{x\in(\partial\Pi)^-_{r/2}\cap\pi_j^-:d(x,\pi_j)<r/2\right\}.
$$
Notice that $(\partial\Pi)^-_{r/2}=\bigcup_j(\partial\Pi)^-_{r/2,j}$ and that, if 
$i\neq j$, $(\partial\Pi)^-_{r/2,i}$ and $(\partial\Pi)^-_{r/2,j}$ have non empty intersection only near the edges of $\Pi$.

Let now $\Sigma_{r/2,j}$ be the projection of $(\partial\Pi)^-_{r/2,j}$ onto $\pi_j$. Notice that 
$\mathcal{H}^{N-1}_\gamma(\Sigma_{r/2,j})\le\mathcal{H}^{N-1}_\gamma(\Sigma_j)+Cr$ for $r$ sufficiently small. We thus infer, by Lemma \ref{lem:upest} and the following Remark \ref{rem:upest}
\begin{align*}
&\int_{(\partial\Pi)^-_{r/2}\setminus\bigcup_{i=1}^{N_r}Q_r(x_i)}d\gamma(x)\int_{\Pi^c\cap\Omega}K_s(x,y)d\gamma(y)\\
&\le\sum_j\int_{(\partial\Pi)^-_{r/2,j}\setminus\bigcup_{i=1}^{N_r}Q_r(x_i)}d\gamma(x)\int_{\Pi^c\cap\Omega}K_s(x,y)d\gamma(y)+C_{N,r}\\
&\le\sum_j\int_{(\partial\Pi)^-_{r/2,j}\setminus\bigcup_{i=1}^{N_r}Q_r(x_i)}d\gamma(x)\int_{\pi_j^+}K_s(x,y)d\gamma(y)+C_{N,r}\\
&\le\frac{2^{\frac s2}}{\pi(1-s)}\sum_j\mathcal{H}^{N-1}_\gamma\left(\Sigma_{r/2,j}\setminus\bigcup_{i=1}^{N_r}D_r(x_i)\right)(1+Cr)+C_{N,r}\\
&\le\frac{2^{\frac s2}}{\pi(1-s)}\mathcal{H}^{N-1}_\gamma\left((\partial\Pi\cap\Omega)\setminus\bigcup_{i=1}^{N_r}D_r(x_i)\right)(1+Cr)+C_{N,r}
\end{align*}
where the constant $C_{N,r}$ estimates the integrals with $|x-y|\ge r/2$ and it possibly changes on each line.
\end{itemize}
By putting together the estimates (A), (B), (C) and by summing the $N_r$ contributes of the cubes $Q_r(x_i)$ we finally get
\begin{equation}\label{eq:stripes}
\begin{split}
&(1-s)P^{\gamma,L}_s(\Pi;\Omega)=(1-s)\int_{\Pi\cap\Omega}d\gamma(x)\int_{\Pi^c\cap\Omega}K_s(x,y)d\gamma(y)
\\
&\le\frac{2^{\frac{s}{2}}}{\pi}\left[\mathcal{H}^{N-1}_\gamma\left((\partial\Pi\cap\Omega)\setminus\bigcup_{i=1}^{N_r}D_r(x_i)\right)+\sum_{i=1}^{N_r}\mathcal{H}^{N-1}_\gamma(D_r(x_i))\right]
+Cr+(1-s)C_{N,r}\\
&=\frac{2^{\frac{s}{2}}}{\pi}P^\gamma(\Pi;\Omega)+Cr+(1-s)C_{N,r}.
\end{split}
\end{equation}
We conclude the proof of this step by letting $s\to 1^-$ and considering the arbitrariness of $r$.

\textbf{Step 2.} Estimate of $P^{\gamma,NL}_s(\Pi;\Omega)$. We have
$$
P^{\gamma,NL}_s(\Pi;\Omega)=\int_{\Pi\cap\Omega}d\gamma(x)\int_{\Pi^c\cap\Omega^c}K_s(x,y)d\gamma(y)
+\int_{\Pi\cap\Omega^c}d\gamma(x)\int_{\Pi^c\cap\Omega}K_s(x,y)d\gamma(y).
$$
Let us fix $\delta>0$ and consider the sets $\Omega^+_\delta$ and $\Omega^-_\delta$. We first estimate
$$
\int_{\Pi\cap\Omega}d\gamma(x)\int_{\Pi^c\cap\Omega^c}K_s(x,y)d\gamma(y)
$$
by splitting in different cases
\begin{itemize}
\item[(A)] $x\in(\Pi\cap\Omega)\setminus\Omega^-_\delta$, $y\in(\Pi^c\cap\Omega^c)\setminus\Omega^+_\delta$, 
\item[(B)] $x\in(\Pi\cap\Omega)\setminus\Omega^-_\delta$, $y\in\Pi^c\cap\Omega^+_\delta$,
\item[(C)] $x\in\Pi\cap\Omega^-_\delta$, $y\in(\Pi^c\cap\Omega^c)\setminus\Omega^+_\delta$, 
\item[(D)] $x\in\Pi\cap\Omega^-_\delta$, $y\in\Pi^c\cap\Omega^+_\delta$.
\end{itemize}

\begin{figure}[h]
\centering
\includegraphics[width=\textwidth]{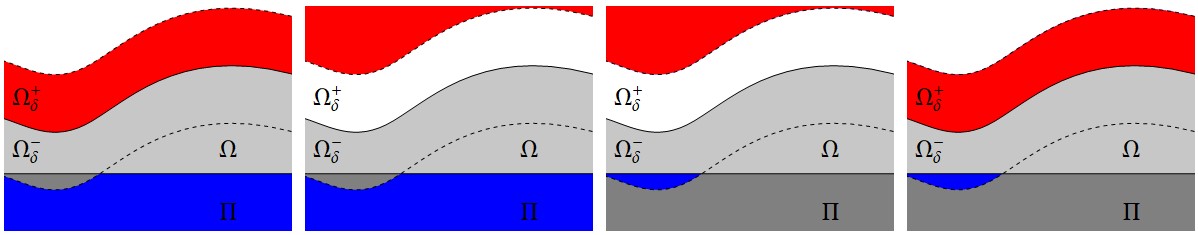}
\caption{Cases (A), (B), (C) and (D). The $x$ variable is in blue, the $y$ variable is in red.}
\end{figure}

Cases (A), (B) and (C) can be treated together, since in that cases $K_s(x,y)$ is uniformly bounded 
(the distance between $x$ and $y$ is larger than $\delta$) by a positive constant $C_{N,\delta}$ depending only on $N$ and $\delta$; in other 
words
\begin{equation}\label{eq:pnl1}
\int_{(\Pi\cap\Omega\times\Pi^c\cap\Omega^c)\setminus((\Pi\cap\Omega^-_\delta)\times(\Pi^c\cap\Omega^+_\delta))}K_s(x,y)d\gamma(x)d\gamma(y)\le C_{N,\delta}.
\end{equation}
In the case (D), we have
\begin{equation}\label{eq:pnl2}
\begin{split}
&\int_{\Pi\cap\Omega^-_\delta}d\gamma(x)\int_{\Pi^c\cap\Omega^+_\delta}K_s(x,y)d\gamma(y)
\\
&\le\int_{\Pi\cap(\Omega^-_\delta\cup\Omega^+_\delta)}d\gamma(x)
\int_{\Pi^c\cap(\Omega^-_\delta\cup\Omega^+_\delta)}K_s(x,y)d\gamma(y)
\le P^{\gamma,L}_s(\Pi;\Omega^-_\delta\cup\Omega^+_\delta).
\end{split}
\end{equation}
By summing \eqref{eq:pnl1} and \eqref{eq:pnl2} and multiplying by $(1-s)$ we get
\begin{equation}\label{eq:pnlfinal}
\begin{split}
&(1-s)\int_{\Pi\cap\Omega}d\gamma(x)\int_{\Pi^c\cap\Omega^c}K_s(x,y)d\gamma(y)
\\
&\le (1-s)C_{N,\delta}+(1-s)P^{\gamma,L}_s(\Pi;\Omega^-_\delta\cup\Omega^+_\delta)
\\
&\le (1-s)C_{N,\delta}+\frac{2^{\frac{s}{2}}}{\pi}P^\gamma(\Pi;\Omega^-_\delta\cup\Omega^+_\delta),
\end{split}
\end{equation}
where the last inequality is a consequence of the first step with $\Omega^-_\delta\cup\Omega^+_\delta$
in place of $\Omega$. By switching $\Pi$ and $\Pi^c$ in \eqref{eq:pnlfinal} and summing up the 
two contributions, we get the thesis.
\end{proof}

\section{Convergence of local minimisers}\label{sec:convoflocmin}

We begin this section generalising \cite[Proposition 16]{AmDeMa} to the radial kernel $\tilde{K}_s$ defined 
in Lemma \ref{lem:upasest}. 

\begin{proposition}
	\label{prop:analogpropsedici}
	Let $\Omega\subseteq\R^N$, $u\in BV(\Omega)$ and $\Omega'\Subset\Omega$. If we set
	$$
	\mathcal{F}_{\tilde{K}_s}(u;\Omega'):=\int_{\Omega'}dx\int_{\Omega'}|u(x)-u(y)|
	\tilde{K}_s(|x-y|)e^{-\frac{|x-y|^2}{4}}dy,
	$$
	we have 
	\begin{equation}
	\label{eq:claimpropsedici}
	\limsup_{s\to 1^-}(1-s)\mathcal{F}_{\tilde{K}_s}(u;\Omega')\le 
	C_N\limsup_{|h|\to 0^+}\int_{\Omega'}\frac{|u(x+h)-u(x)|}{|h|}dx\le 
	C_N|Du|(\Omega).
	\end{equation}
\end{proposition}
\begin{proof}
Observe immediately that the second inequality is well known, see e.g. \cite[Remark 3.25]{AFP}, hence the
central $\limsup$ is finite. For $h\in\R^N$ we define
	$$
	g(h):=\int_{\Omega'}\frac{|u(x+h)-u(x)|}{|h|}dx
	$$
	and fix $L>\limsup_{|h|\to 0^+}g(h)$. Then there exists $\delta_L>0$ such that $\Omega'+h\subset\Omega$ for all $h\in B_{\delta_L}$ and 
	\begin{equation}
	\label{eq:dav}
	L\ge g(h)\quad\text{for any}\quad 0<|h|\le\delta_L.
	\end{equation}
	Multiplying both sides of \eqref{eq:dav} by $|h|e^{-\frac{|h|^2}{4}}\tilde{K}_s(|h|)$ and integrating 
	with respect to the variable $h$ on $B_{\delta_L}$ we have
	\begin{align}
	\label{eq:trentacinque}
	L\int_{B_{\delta_L}}&|h|\tilde{K}_s(|h|)e^{-\frac{|h|^2}{4}}dh\ge
	\int_{B_{\delta_L}}g(h)|h|\tilde{K}_s(|h|)e^{-\frac{|h|^2}{4}}dh
	\\ \nonumber
	&=
	\int_{B_{\delta_L}}e^{-\frac{|h|^2}{4}}dh\int_{\Omega'}|u(x+h)-u(x)|\tilde{K}_s(|h|)dx.
	\end{align}
	Moreover, summing up estimates \eqref{eq:stimapiccola} and \eqref{eq:stimagrande} with $\alpha=1$ and $R_\Omega=\delta_L$, we have 
	\begin{equation}
	\label{eq:trentacinqueprimo}
	\begin{split}
	&LN\omega_N\left(\frac{\delta_L^{N+1}}{N+1}\frac{2}{s(1-e^{-2})}+
	2^{\frac{N+3}{2}}\Gamma\left(\frac{N+1}{2}\right)\frac{e^{\frac{N+1}{2}}}{1-s}\right)
\\
&\geq LN\omega_N\int_0^{\delta_L}r^N\tilde{K}_s(r)dr=L\int_{B_{\delta_L}}|h|\tilde{K}_s(|h|)dh\ge L\int_{B_{\delta_L}}|h|\tilde{K}_s(|h|)e^{-\frac{|h|^2}{4}}dh.
	\end{split}
	\end{equation}
	Now we notice that
	\begin{equation}
	\label{eq:trentasei}
	\begin{split}
	\mathcal{F}_{\tilde{K}_s}(u;\Omega')=&
	\iint_{\{\Omega'\times\Omega'\cap|x-y|\le\delta_L\}}|u(x)-u(y)|\tilde{K}_s(|x-y|)
	e^{-\frac{|x-y|^2}{4}}dxdy
	\\
	&+\iint_{\{\Omega'\times\Omega'\cap|x-y|>\delta_L\}}|u(x)-u(y)|\tilde{K}_s(|x-y|)
	e^{-\frac{|x-y|^2}{4}}dxdy \\
	=&\int_{B_{\delta_L}}e^{-\frac{|h|^2}{4}}dh\int_{\Omega'}|u(x+h)-u(x)|\tilde{K}_s(|h|)dx
	\\
	&+\int_{B^c_{\delta_L}}e^{-\frac{|h|^2}{4}}dh\int_{\Omega'}|u(x+h)-u(x)|\tilde{K}_s(|h|)dx 
	\\
	\leq &\int_{B_{\delta_L}}e^{-\frac{|h|^2}{4}}dh\int_{\Omega'}|u(x+h)-u(x)|\tilde{K}_s(|h|)dx
	\\
	&+2N\omega_N\left\|u\right\|_{L^1(\Omega)}\int_{\delta_L}^{\infty}r^{N-1}\tilde{K}_s(r)
	e^{-\frac{r^2}{4}}dr,
	\end{split}
	\end{equation}
	where the second term in the right-hand side in \eqref{eq:trentasei} is finite thanks to Lemma \ref{lem:upasest}.
	
	To conclude, putting together \eqref{eq:trentacinqueprimo}, \eqref{eq:trentacinque} and \eqref{eq:trentasei}, multiplying by $(1-s)$ and passing to the $\limsup$ for $s\to 1^-$ we obtain
	$$
	C_NL\ge\limsup_{s\to 1^-}(1-s)\mathcal{F}_{\tilde{K}_s}(u;\Omega'),
	$$
	and for $L\rightarrow\limsup_{|h|\to 0^+}g(h)$ we have proved the first inequality in \eqref{eq:claimpropsedici}. 
	\end{proof}

We now prove that if a sequence $(E_n)$ of local minimisers of $P^\gamma_{s_n}(\cdot;\Omega)$ converges to $E$ and 
$s_n\uparrow 1$ then $E$ is a local minimiser of $P^\gamma(\cdot;\Omega)$. Recall that a set $E$ is a local minimiser of $P^\gamma_s(\cdot;\Omega)$ if $P^\gamma_s(E;\Omega)\leq P^\gamma_s(F;\Omega)$ whenever 
$E\triangle F \Subset \Omega$.

\begin{theorem}[\textbf{Convergence of local minimisers}]
	Let $(s_n)_{n\in\N}$ be a sequence in $(0,1)$ such that $s_n\uparrow 1$ and, for any $n\in\N$, let $E_n$ be a local minimiser of $P^\gamma_{s_n}(\cdot;\Omega)$ such that $\chi_{E_n}\rightarrow\chi_E$ in $L^1_{\text{\rm loc}}(\R^N)$. Then
	\begin{equation}
	\label{eq:limsupfinite}
	\limsup_{n\to\infty}(1-s_n)P^\gamma_{s_n}(E;\Omega')<\infty\quad\forall\Omega'\Subset\Omega,
	\end{equation}
	the limit set $E$ is a local minimiser of $P^\gamma(\cdot;\Omega)$ and 
	$(1-s_n)P^\gamma_{s_n}(E_n;\Omega')\rightarrow \frac{\sqrt{2}}{\pi} P^\gamma(E;\Omega')$ as $n\to\infty$ for any $\Omega'\Subset\Omega$ such that $P(E;\partial\Omega')=0$.
\end{theorem}
\begin{proof}
	We firstly prove \eqref{eq:limsupfinite}.
	Thanks to Proposition \ref{prop:analogpropsedici} with $u=\chi_{E_n}$ and to  
	\eqref{eq:upperadialestimate}, racalling the measure $\lambda$ defined in \eqref{deflambda} we have 
	$$
	\limsup_{n\to\infty}(1-s_n)P^{\gamma,L}_{s_n}(E_n;\Omega')\le
	\limsup_{n\to\infty}(1-s_n)P^{\lambda,L}_{\tilde{K}_{s_n}}(E_n;\Omega')<\infty
	$$
	and
	$$
	\limsup_{n\to\infty}(1-s_n)P^{\gamma,NL}_{s_n}(E_n;\Omega')\le
	2\limsup_{n\to\infty}(1-s_n)P^\lambda_{\tilde{K}_{s_n}}(\Omega')<\infty.
	$$
	Now we prove the second part of the claim for compactly supported balls $B_R(x)$ in $\Omega$. The 
	extension to general $\Omega'\Subset\Omega$ goes as in \cite{AmDeMa}. Since in the 
	sequel there is no ambiguity, for any $\varrho>0$ we denote $B_\varrho(x)$ simply with $B_\varrho$. 
	Consider the monotone set function $\alpha_n(A)=(1-s_n)P^{\gamma,L}_{s_n}(E_n;A)$ for every open set 
	$	A\subset\Omega$, and extended to any $F$ by 
	$$
	\alpha_n(F):=\inf\{\alpha_n(A);\quad F\subset A\subset\Omega,\quad A\quad \text{open}\}.
	$$
	Clearly, $\alpha_n$ is regular (see definition in \cite[Pag. 23]{AmDeMa}). Thanks to \eqref{eq:limsupfinite} and to 
	the De Giorgi-Letta Theorem (see \cite[Theorem 1.53]{AFP}), the sequence $\alpha_n$ weakly converges to 
	a regular 
	monotone and superadditive set function $\alpha$. We now prove that if $B_R\Subset\Omega$ and 
	$\alpha(\partial B_R)=0$, then $E$ is a local minimiser of the functional $P^\gamma(\cdot;B_R)$ and
	$$
	\lim_{n\to\infty}(1-s_n)P^\gamma_{s_n}(E_n;B_R)=P^\gamma(E;B_R).
	$$
	Indeed, let $F\subset\Omega$ be a Borel set such that $E\triangle F\Subset B_R$; then, there exists $r<R$
	such that $E\triangle F\subset B_r$. By the $\Gamma$-limsup inequality \eqref{eq:gammalimsup} there 
	exists a sequence $(F_n)$ such that
	$$
	\lim_{n\to\infty}|(F_n\triangle F)\cap B_R|=0\quad\text{and}\quad
	\lim_{n\to\infty}(1-s_n)P^\gamma_{s_n}(F_n;B_R)=\frac{\sqrt{2}}{\pi}P^\gamma(F;B_R).
	$$
	According to Proposition \ref{prop:gluing}, given $\varrho$ and $t$ such that $r<\varrho<t<R$, we can find sets $G_n$ such that for any $n\in\N$
	$$
	G_n=E_n\quad\text{in}\quad\R^N\setminus B_t,\quad\quad G_n=F_n\quad\text{in}\quad B_\varrho
	$$
	and for any $\varepsilon>0$ the inequality
	\begin{equation*}
	\begin{split}
	P^{\gamma,L}_s(G_n;B_R)\le& P^{\gamma,L}_s(F_n;B_R)
	+P^{\gamma,L}_s(E_n;B_R\setminus\overline{B}_{\varrho-\varepsilon})+2^{N+1}\tilde{K}_s(\varepsilon) 
	\\
	&+C'(N,R,\delta_1,\delta_2)\left(\frac{c_N}{1-s}
	+\frac{c'(R,N)}{s}\right)\left\|\chi_{E_n}-\chi_{F_n}\right\|_{L^1_\lambda(B_t\setminus B_\varrho)} 
	\\
	&+C'(N,R,\delta_1,\delta_2)\left(\frac{d_N}{2-s}+
	\frac{d'(R,N)}{s}\right)\left\|\chi_{E_n}-\chi_{F_n}\right\|_{L^1_\lambda(B_R)}
	\end{split}
	\end{equation*}
	holds. By the local minimality of $E_n$ we infer
	$$
	P^\gamma_{s_n}(E_n;B_R)\le P^\gamma_{s_n}(G_n;B_R).
	$$
	We now estimate
	$P^{\gamma, NL}_{s_n}(G_n;B_R):=I+II$, see \eqref{defNonlocal}.	We have
	\begin{equation*}
	\begin{split}
	I=&\int_{G_n\cap B_t}d\gamma(y)\int_{E_n^c\cap B_R^c}K_{s_n}(x,y)d\gamma(x)
	+\int_{E_n\cap (B_R\setminus B_t)}d\gamma(y)\int_{E_n^c\cap B_R^c}K_{s_n}(x,y)d\gamma(x)
	\\
	\le& 2^N\tilde{K}_{s_n}(R-t) 
	+\int_{E_n\cap (B_R\setminus B_t)}d\gamma(y)\int_{E_n^c\cap (B_{R'}\setminus B_R)}K_{s_n}(x,y)d\gamma(x)
	\\
	&+\int_{E_n\cap (B_R\setminus B_t)}d\gamma(y)\int_{E_n^c\cap B_{R'}^c}K_{s_n}(x,y)d\gamma(x)
	\\
	\le& 2^N\tilde{K}_{s_n}(R-t) 
	+\int_{E_n\cap (B_{R'}\setminus \overline{B}_t)}d\gamma(y)
	\int_{E_n^c\cap (B_{R'}\setminus \overline{B}_t)}K_{s_n}(x,y)d\gamma(x)
	\\
	&+\int_{E_n\cap (B_R\setminus B_t)}d\gamma(y)\int_{E_n^c\cap B_{R'}^c}K_{s_n}(x,y)d\gamma(x)
	\\
	\le& P^{\gamma,L}_{s_n}(E_n;B_{R'}\setminus \overline{B}_t)
	+2^N\left(\tilde{K}_{s_n}(R-t)+\tilde{K}_{s_n}(R'-R)\right),
	\end{split}
	\end{equation*}
	for any $R'\in(R,d(x,\partial\Omega))$. Since $II$ can be estimated in an analogous way we have 
	$$
	P^{\gamma, NL}_{s_n}(E_n;B_R)\le 2P^{\gamma,L}_{s_n}(E_n;B_{R'}\setminus \overline{B}_t)
	+2^{N+1}\left(\tilde{K}_{s_n}(R-t)+\tilde{K}_{s_n}(R'-R)\right),
	$$
	and so
	$$
	\limsup_{n\to\infty}(1-s_n)P^{\gamma, NL}_{s_n}(E_n;B_R)\le
	2\limsup_{n\to\infty}(1-s_n)P^{\gamma,L}_{s_n}(E_n;B_{R'}\setminus \overline{B}_t).
	$$
	Finally
	\begin{equation}
	\label{eq:trentadue}
	\begin{split}
	\frac{\sqrt{2}}{\pi}P^\gamma(E;B_R)\le&\liminf_{n\to\infty}(1-s_n)P^{\gamma,L}_{s_n}(E_n;B_R)
	\le\liminf_{n\to\infty}(1-s_n)P^\gamma_{s_n}(E_n;B_R) 
	\\
	\le&\liminf_{n\to\infty}(1-s_n)P^\gamma_{s_n}(G_n;B_R) 
	\\
	\le&\liminf_{n\to\infty}(1-s_n)P^{\gamma,L}_{s_n}(G_n;B_R)
	+\limsup_{n\to\infty}(1-s_n)P^{\gamma,NL}_{s_n}(G_n;B_R) 
	\\
	\le&\liminf_{n\to\infty}(1-s_n)P^{\gamma,L}_{s_n}(F_n;B_R)
	+3\limsup_{n\to\infty}(1-s_n)P^{\gamma,L}_{s_n}(E_n;B_{R'}\setminus\overline{B}_{\varrho-\varepsilon}) 
	\\
	&+C\lim_{n\to\infty} |(E_n\triangle F_n)\cap (B_t\setminus B_\varrho)|.
	\end{split}
	\end{equation}
	The last limit is zero, as $E=F$ in $B_t\setminus B_\varrho$ and $|(E_n\triangle E)\cap B_R|\to 0$, 
	$|(F_n\triangle F)\cap B_R|\to 0$ as $n\to\infty$. Using \cite[Proposition 22]{AmDeMa}, and recalling that $\alpha(\partial B_R)=0$, we infer
	$$
	\lim_{R'\to R,\varrho\to R, \varepsilon\to 0}\limsup_{n\to\infty}(1-s_n)P^{\gamma,L}_{s_n}(E_n;B_{R'}\setminus\overline{B}_{\varrho-\varepsilon})=\lim_{\delta\to 0}\limsup_{n\to\infty}\alpha_n(B_{R+\delta}\setminus \overline{B}_{R-\delta})=0,
	$$
	and finally \eqref{eq:trentadue} yields
	$$
	\frac{\sqrt{2}}{\pi}P^\gamma(E;B_R)=\lim_{n\to\infty}(1-s_n)P^\gamma_{s_n}(F_n;B_R)=\frac{\sqrt{2}}{\pi}P^\gamma(F;B_R).
	$$
	Therefore $E$ is a local minimiser of $P^\gamma(\cdot;B_R)$. Choosing $F=E$ the inequalities in \eqref{eq:trentadue} become
	$$
	\lim_{n\to\infty}(1-s_n)P^\gamma_{s_n}(E_n;B_R)=\lim_{n\to\infty}(1-s_n)P^{\gamma,L}_{s_n}(E_n;B_R)=\frac{\sqrt{2}}{\pi}P^\gamma(E;B_R).
	$$
	\end{proof}

\paragraph*{\bf Funding}\ \\
A.C. has been partially supported by the TALISMAN project Cod. ARS01-01116. S.C. has been partially supported by the ACROSS project Cod. ARS01-00702.
D.A.L. has been supported by the Academy of Finland grant 314227. 
D.P. is member of G.N.A.M.P.A. of the Italian Istituto Nazionale di Alta Matematica (INdAM) and has been partially supported by the PRIN 2015 MIUR project 2015233N54.

\end{document}